\documentclass[11pt,a4paper]{article}
\usepackage{amsmath,amsthm,graphicx,rotating,upgreek}
\usepackage{amssymb,bbm}
\usepackage[utf8]{inputenc}
\usepackage[english]{babel}
\usepackage{xcolor}
\usepackage{hyperref}
\usepackage{comment}
\hypersetup{
    colorlinks=true,
    linkcolor=blue,
    filecolor=magenta,      
    urlcolor=cyan,
    }
\usepackage{ulem}
\newtheorem{theorem}{Theorem}[section]
\newtheorem{lemma}[theorem]{Lemma}

\newtheorem{proposition}[theorem]{Proposition}

\newtheorem{remark}[theorem]{Remark}
\newtheorem{definition}[theorem]{Definition}
\newtheorem{definition-proposition}[theorem]{Definition-Proposition}
\newtheorem{hypothesis}[theorem]{Hypothesis}
\newtheorem{example}[theorem]{Example}

\numberwithin{equation}{section}

\newcommand\Tr{\mathrm{Tr}}
\newcommand\esp{\mathbb E}

\newcommand\limN{\underset{N \rightarrow \infty}\longrightarrow}
\newcommand\Nlim{\underset{N \rightarrow \infty}\lim}
\newcommand\eqa{\begin{eqnarray}}
\newcommand\qea{\end{eqnarray}}
\newcommand\eq{\begin{eqnarray*}}
\newcommand\qe{\end{eqnarray*}}

\newcommand\etc{,\ldots ,}

\newcommand\one{\mathbbm{1}}
\newcommand\mbf{\mathbf}
\newcommand\mcal{\mathcal}
\newcommand\mbb{\mathbb}
\newcommand\mrm{\mathrm}

\title{A traffic approach for profiled Pennington-Worah matrices}

\author {Issa Dabo and Camille Male}

\begin{document}
\date{}
\maketitle
\begin{abstract}    We study macroscopic observables of large random matrices introduced by Pennington and Worah, of the form $Y(h) = \frac 1 {\sqrt{N_2}} h \big[ \big\{ \frac{ W X}{\sqrt{N_0}} \big\} \big]$, where $W$ and $X$ are random rectangular matrices with independent entries and $h$ is a function evaluated entry-wise. We allow the variance of the entries of the matrices to vary from entry to entry. We complement Péché perspective from [Electron. Commun. Probab. 24 (2019), no. 66, 1–7] showing a decomposition of $Y(h)$ whose and traffic asymptotic traffic-equivalent for their ingredients, when $h$ belong to the space of odd polynomials. This give a new interpretation of the "linear plus chaos" phenomenon observed for these matrices. 
\end{abstract}

Primary 15B52, 46L54; Keywords: 
Free Probability, Large Random Matrices

\setcounter{tocdepth}{2}

{\bf Notations:}  For an integer $n\geq 1$, we use the notation $[n]:=\{1\etc n\}$. For a real rectangular matrix $A$ and a function $h:\mbb R \to \mbb R$, we denote by $h[\{A\}]$ the entry-wise evaluation of $h$ in $A$, that is the matrix whose entries are the image by $h$ of the corresponding entries of $A$. We consider matrix sizes $N_0,N_1, \dots$ for rectangular matrices in the classical regime of random matrices, i.e. for each $i \geq 1$, we implicitly assume that $N_i$ is a sequence $N_i(N)$ for a parameter $N\geq 1$ that tends to infinity, and such that the ratio $ \frac{N_i}N $ converges to a positive limit when $N\rightarrow \infty$.

\section{Introduction}\label{Sec:Intro}

In this article, we study random Gram matrices that were introduced by Jeffrey Pennington and Pratik Worah in the context of machine learning \cite{PenWor}. We refer the reader to  \cite{Mont22} for the motivation of the authors in this context.

\begin{definition}\label{MatrixModel}
Given two complex random matrices $W$ in $\mbb R^{N_1 \times N_0}$  and $X$ in $\mbb R^{N_0 \times N_2}$, as long as a  function $h: \mbb R \to \mbb R$, we set
	\eqa \label{Eq:MatrixModel}
		Y(h)= \frac 1 {\sqrt{N_2}} h \left[ \left\{ \frac{ W X}{\sqrt{N_0}} \right\} \right] \in \mbb R^{N_1 \times N_2},
	\qea
that we call the \emph{Pennington-Worah matrix} associated to $h,W,X$. 
\end{definition}

The weak convergence of the empirical eigenvalues distribution of $YY^t$ for such a matrix $Y$ is proved in \cite{PenWor} for Gaussian entries. Lucas Benigni and Sandrine Péché \cite{BenPec} extends their result and study the outliers under the following hypotheses.

\begin{hypothesis}\label{WignerLikeModel} The rectangular matrices $W$ and $X$ are independent and have centered real i.i.d. entries with variance one. Moreover there exist constants $\vartheta>0$, $\alpha>1$ such that
	\eq\label{Eq:BoundEntries}
	 \mbb P\big( |W(1,1)| \geq t \mrm{ \ and \ }  |X(1,1)| \geq t \big) \leq e^{-\vartheta t^\alpha}.
	 \qe
The function $h:\mbb R \to \mbb R$ is real analytic, and there exist constants $C,c,A_0>0$ such that $|h^{(n)}(x)|\leq CA^{cn}$, for all $A\geq A_0$,   $n\in \mbb N$ and   $x\in [-A,A]$.  
\end{hypothesis}

Under the above assumptions, the empirical eigenvalues distribution of $YY^*$ converges in probability toward a deterministic limit has $N$ goes to infinity. The limit is described by a consistent system of equations for Stieltjes transforms. Moreover, in \cite{Peche19}  Sandrine P\'ech\'e proposes a presentation of this distribution by exhibiting a \emph{simple equivalent model} with same limiting distribution. She also use this method with Lucas Bengnini in \cite{BenPec22} to describe the outliers.

To state their results from \cite{Peche19,BenPec22}, we recall that a sequence of sets $S_N\subset \mbb R$ converges in Hausdorff topology toward $S \subset \mbb R$ whenever for $\min\{ \varepsilon >0 \, | \, S_N+(-\varepsilon, \varepsilon) \subset S, S+(-\varepsilon, \varepsilon) \subset S_N\}$ tends to zero. We use the following terminology.

\begin{definition} Let $A=A_N$ and $B=B_N$ be two sequences for $N_1\times N_2$ rectangular matrices, where $\frac {N_i}N\limN \psi_i>0$ for $i=1,2$. We say that $A$ and $B$ are \emph{spectral equivalent} if the empirical eigenvalues distributions of $AA^t$ and $BB^t$ converges to the same limit. We say that $A$ and $B$ are \emph{strongly spectral equivalent} if moreover the spectra of $AA^t$ and $BB^t$ converges to a same set in Hausdorff topology.
\end{definition}

Strongly spectral equivalent matrices have outliers converging to the same positions \cite[Proposition 2.1]{CollinsMale}. We denote by $\omega(t) = \frac{ e^{-t^2/2}}{\sqrt{2\pi}}$ the density of a real standard Gaussian random variable. A random matrix is say to be a \emph{standard Gaussian matrix} it is has i.i.d. real standard Gaussian entries.

\begin{theorem}\label{Th:EtatArt} A random  $Y(h)$ as in Definition \ref{MatrixModel} satisfying Hypothesis \ref{WignerLikeModel} is  spectral equivalent to  
	$$ Y^{\square\mrm{lin}}(h)  + Y^{\square\mrm{per}}(h),$$
where $Y^{\square\mrm{lin}}(h) = \sqrt{\theta_2(h)} \frac{ W^{\mrm{Gau}} X^{\mrm{Gau}}}{\sqrt{N_0N_2}}$ and $Y^{\square\mrm{per}}(h) =  \sqrt{\theta_1(h)-\theta_2(h)} \frac{ Z^{\mrm{Gau}}}{\sqrt{N_0}}$.
The matrices $W^{\mrm{Gau}}, X^{\mrm{Gau}}$ and $Z^{\mrm{Gau}}$ are independent standard Gaussian matrices and we have set
	\eq
		\theta_1(h) = \int_\mbb R h^2(t)  \omega(t) \mrm d t , &\quad& \theta_2(h) = \Big(\int_\mbb R h'(t) \omega(t) \mrm d t\Big)^2.
	\qe
Moreover, if the third moment of the entries of $W$ and $X$ is zero, then $Y(h)$ is strongly spectral equivalent to 
	$$Y^{\square}(h)=Y^{\square\mrm{lin}}(h)  + Y^{\square\mrm{per}}(h) +Y^{\square\mrm{def}}(h),$$
where $Y^{\square\mrm{def}}(h)$ is an explicit matrix of rank 2.
\end{theorem}

In words, the first statement says that $Y(h)$ is spectral equivalent to the matrix when $h$ is linear plus an independent i.i.d matrix. We refers this as the {\it linear plus chaos} phenomenon. Our work is motivated by the following questions:
\begin{enumerate}
	\item {\bf Stability of the phenomenon.} Do we still have an analogue linear plus chaos phenomenon when the matrices $W$ and $X$ are replaced by more general models with more structure ?  In this article, we consider \emph{profiled matrices}, namely matrices with independent entries where the variance of the entries can varies from one variable to another, see Hypothesis \ref{ProfiledModel}. Working with a different type of equivalent we confirm that the phenom holds.
	\item {\bf Structure of the noise.} Let $h_1$ and $h_2$ be two functions as above, for which both $Y(h_1)$ and $Y(h_2)$ satisfy the linear plus chaos phenomenon. What can be said about the joint distribution of these noises ? Can we find functions for which these noise are independent ? Or functions for which they are coupled ? The roots of this question lie in the very origin of free probability theory which questions the difference between probability spaces generated by different numbers of free random variables \cite{Dyk94}. This article completely characterize a family of independent matrices that generate the noise arising from profiled Pennington-Worah matrices.
	\item {\bf Description of the phenomenon.} What is the intrinsic reason for which Pennington-Worah matrices to exhibit such a simple behavior ? To progress in this question, we propose in Section [X] a decomposition of a Pennington-Worah as the sum of 
		$$Y(h) = Y^{\mrm{lin}} + \sum_{m\geq 2} Y_m^{\mrm{per}}(h) + Y^{\mrm{def}}(h) + \epsilon(h),$$
an we state in Section X the joint convergence of each ingredients of the above sum toward an ingredient of a linear plus chaos decomposition. Our interpretation of the emerges of chaos is presented in Remark X.
\end{enumerate}

\section{Hypotheses and definition of traffic equivalence}

\subsection{Matrix model and decomposition}

\subsubsection{Model of variance profiled matrices}
We present in this section our model. 

\begin{hypothesis}\label{ProfiledModel} The two random rectangular matrices $W$ and $X$ can be written
	\eq
	W = \Gamma_w \circ W', \quad X = \Gamma_x \circ X'
	\qe
where $\circ$ denotes the entry-wise product of matrices.
	\begin{enumerate}
	\item The matrices $W$ and $X$ are respectively of size $N_1\times N_0$ and $N_0\times N_2$, and setting $N=N_0+N_1+N_2$ the sequences defined by $\psi_i = \frac{N_i}N,i=0,1,2$, converge to positive number.
	\item The matrices of $  W'$ and $   X'$ independent and have centered real i.i.d. entries with variance one, the laws of their entries do not depend on $N$ and have finite moments of all order. 
	\item The entries of $\Gamma_w$ and $\Gamma_x$ are bounded, and $(\Gamma_w, \Gamma_x)$ converges in graphons topology, see Definition \ref{GraphonsDistrCarre} below.	
This holds if $\Gamma_w= \big( \gamma_w(\frac{i}N, \frac{j}N) \big)_{i,j}$ and $\Gamma_x=\big( \gamma_x(\frac{i}N, \frac{j}N) \big)_{i,j}$, where $\gamma_w$ and $\gamma_x$ are piecewise continuous maps $[0,1]^2\to \mbb R$. We call $\Gamma_w$ and $\Gamma_x$ the \emph{variance profiles} (or simply profiles) of $w$ and $x$. 
\end{enumerate}
\end{hypothesis}

\subsubsection{Hermite polynomials}
We recall the special role of Hermite polynomials $(g_n)_{n\geq 0}$ to understand Theorem \ref{Th:EtatArt}. Recall that $\omega(t) = \frac{ e^{-t^2/2}}{\sqrt{2\pi}}$ and
	$$ g_n : t \mapsto   \frac{({-1})^n\omega^{(n)}(t)}{\omega(t)}, \quad n\geq 0,$$
where $\omega^{(n)}$ denotes the $n$-th derivative of $\omega$. The collection $(g_n)_{n\geq 0}$ is an orthonormal basis of the space of $\mbb C[y]$ with respect to the standard Gaussian law, normalized such that
	$ \int_\mbb R   g_n g_m   \mrm d\omega = \delta_{n,m} n!.$
The symbol $\delta_{n,m}$ stands for the usual Kronecker symbol. For any $n\geq 1$, we have $g'_n = n g_{n-1}, \quad \forall n\geq 1$ so in particular,  $\int_\mbb R    g_n^{(m)}(t) \omega(t) \mrm d t = \delta_{n,m}n!$ for any $m,n\geq 1$. Hence, for any $h:\mbb R\to \mbb R$ in $\mrm{L}^2(\mrm d\omega)$, 
	$$\frac{ \int h^{(n)}(t) \omega(t) \mrm d t}{n!}$$
is the coefficient of to the $n$-th Hermite polynomial in the basis $(g_n)_{n\geq 0}$.

 For any $h:\mbb R \to \mbb R$, let $\tilde h := \frac{ h- \sqrt \theta g_1}{\sqrt{ 1 - \theta}}$ the orthogonal projection of $h$ on the orthogonal of $g_1$ with respect to $\mrm d \omega$. Since $Y(h)$ is linear in $h$, it is the sum of the matrices
	$$ Y(h) = \sqrt{ \theta} Y(g_1) + \sqrt{ 1- \theta} Y(\tilde h).$$
From Theorem \ref{Th:EtatArt}, $Y(g_1)$ is spectral equivalent to $ \frac{ W^{\mrm{Gau}} X^{\mrm{Gau}}}{\sqrt{N_0N_2}}$ and $Y(\tilde h)$ is spectral equivalent to $\frac{Z^{\mrm{Gau}}}{\sqrt{N_0}}$. The next section introduces a more precise strategy to decompose the matrix, where the Hermite coefficients show up naturally.

\subsubsection{A decomposition of Pennington-Worah matrices}

Let $Y(h )= \gamma h\big[\{ \gamma_0 W X \} \big]$ be a Pennington-Worah matrix, where $W,X$ are rectangular matrices and  $\gamma, \gamma_0>0$. For the polynomial function $h_n:x\mapsto x^n$, the matrix-entry definition gives the expression
	$$Y(h_n) = \Big( \gamma \gamma_0^n  \sum_{\mbf d \in [N_0]^n} \prod_{\ell=1}^n W(i,d_\ell) X(d_\ell,j)\Big)_{\substack{ i=1\etc N_1 \\ j=1\etc N_2}}.$$
Our decomposition involves several definitions.

\begin{definition}\label{Def:LambdaMatrices}
\begin{enumerate}
	\item A set partition of a set $\mcal X$, simply called a partition, is a set of non-empty subsets of $\mcal X$, called its blocks, whose union is $\mcal X$. We denote by $\mcal P(\mcal X)$ the set of partitions of $\mcal X$. Moreover, for any multi-index $\mbf k=(k_1\etc k_n)$, we denote by $\ker {\mbf k}$ the set partition of $[n]:=\{1\etc n\}$ such that $p \overset{\ker \mbf k} \sim q$ if and only if $k_p = k_q$. 
	\item An  integer partition of an integer $n\geq 1$ is a non-increasing tuple  $\lambda^{(n)}  = (\lambda_1\etc \lambda_{\ell})$ of integers, called its parts, which sum up to $n$. We denote $\lambda^{(n)} \vdash n$ to say that $\lambda^{(n)}$ is an integer partition of $n$. Moreover, a set partition $\pi$ of $[n]$ is said to be of type $\lambda^{(n)}$ whenever $\lambda^{(n)}$ is its sequence of blocks size. 
	\item For any integer partition $\lambda^{(n)}$ of $n$ and any partition $\pi_0$ of type $\lambda^{(n)}$, we set
	\eq
		Z(\lambda^{(n)}) := \Big( \sum_{\substack{ \mbf d \mrm{ \, s.t. \, } \\ \ker(\mbf d)=\pi_0}} \prod_{\ell=1}^n W(i,d_\ell) X(d_\ell, j) \Big)_{\substack{ i=1\etc N_1 \\ j=1\etc N_2}}.
	\qe
\end{enumerate}
\end{definition}
The above expression clearly depends only on the type of $\pi_0$. Therefore we have a canonical decomposition 
	$$Y(h_n) =  \sum_{\lambda^{(n)} \vdash n}  c_{\lambda^{(n)}}  \times \gamma \gamma_0^nZ( \lambda^{(n)}),$$
where $c_{\lambda^{(n)}}$ is the number of $\pi \in \mcal P(n)$ of type $\lambda^{(n)}$. 
In particular  for $ \lambda_2^{(n)}:=(2 \etc 2)\vdash n$, then $c_{\lambda^{(n)}}$ is the number of pair partitions of $[n]$, which is equal to $ \int_{\mbb R}  h_n( t) \mrm d \omega (t)$. Moreover, for any $n\geq 1$, with $ \lambda_1^{(n)}:=(2 \etc 2, 1)\vdash n$ and $ \lambda_m^{(n)}:=(m,2 \etc 2)\vdash n$, for $m\geq 2$, we have
	$$c_{\lambda^{(n)}_m} =  \binom{n}m \int_{\mbb R} t^{n-m} \mrm d \omega (t) = \frac 1 {m!} \int_{\mbb R}  h_n^{(m)}( t) \mrm d \omega (t), \quad m\neq 2,$$
 where $h_n^{(m)}$ is the $m$-th derivative of $m$. For any $1<m\leq n$ we denote $ \lambda_1^{(m,n)}=(2 \etc 2, 1\etc 1)\vdash n$ the integer partition with $m$ parts equal to one, for which we also have $ c_{\lambda_1^{(m,n)}} = c_{\lambda^{(n)}_m}$.
\begin{definition}\label{APRIORI} For any $h = \sum_{n\geq 1} a_n h_n$, odd polynomial or an analytic function satisfying Hypothesis \ref{WignerLikeModel}, we denote 
	$$Y(h) =Y^{\mrm{lin}}(h)+ \sum_{m\geq 1} Y^{\mrm{per}}_m(h) +  Y^{\mrm{def}}(h) +\epsilon(h),$$
	where we have set
	\eq
		Y^{\mrm{lin}}(h)& =&   \sum_{n\geq 1} a_n \Big(  \int_{\mbb R}  h_n'( t) \mrm d \omega (t) \Big) \times   \gamma \gamma_0^{n}Z( \lambda_1^{(n)}),\\
		Y^{\mrm{per}}_m(h)& = &  \sum_{n\geq 1} \frac{a_n}{m!} \Big(  \int_{\mbb R}  h_n^{(m)}( t) \mrm d \omega (t) \Big) \times  \gamma \gamma_0^{n}Z( \lambda_1^{(m,n)}), \quad \forall m\geq 2\\
		Y^{\mrm{def}}(h) &=& \sum_{n\geq 1} \frac {a_n} 6  \Big(  \int_{\mbb R}  h_n'''( t) \mrm d \omega (t) \Big)\times  \gamma \gamma_0^{n}Z( \lambda_3^{(n)}),\\
		\epsilon(h) & =&   \sum_{n\geq 1} \sum_{\substack{ \lambda^{(n)} \vdash n \\ \lambda^{(n)}\neq  \lambda_3^{(n)}, \\ \lambda^{(n)}\neq \lambda_1^{(m,n)} \, \forall m\geq 1  }} a_n c_{\lambda^{(n)}}\times  \gamma \gamma_0^{n} Z(\lambda^{(n)}).
	\qe
\end{definition}
Our main result  shows that an asymptotic equivalent for these matrices in a sense that is clarified next section, for which the weights can be computed by simple asymptotic rules for the matrices $Z(\lambda^{(n)}) $. 

\subsection{Traffic equivalent}

A difficulty arise to consider the decomposition of the previous section from the algebraic aspect. We shall go beyond free probability and use the notion of \emph{traffic equivalent} to fit the nature of the Pennington-Worah matrix decomposition, by introducing a generalization of non-commutative polynomials. Although  this notion is not a rigorously an intermediate notion between spectral and strong spectral equivalence, the reader can skip this section with this idea in mind, without major consequence for the understand of next section (note that our approach does not prove the convergence of outliers but gives a proposal for the matrix deformation).

A graph is a couple $(V,E)$ where $V$ is non empty set called the vertex set, and $E$ is a multi-ensemble of couples of elements of $V$, possibly empty, called the edge set. Multi-ensemble means that each element appear with a given multiplicity. The graph are directed: for $e=(v,w)\in E$, we call $v$ the source of $e$, $w$ its target. 

\begin{definition}\label{TestGraphCarre} Let $\Omega$ be a  label set and $\mbf x=(x_\omega)_{\omega \in \Omega}$ a collection of formal variables.
\begin{enumerate}
	\item A \emph{test graph} labeled by $\Omega$ (or in the variables $\mbf x$) is a triplet $T=(V,E,\gamma)$ where $(V,E)$ is a graph,  and $\gamma:E \to \Omega$ is a map associating the variable $x_{\gamma(e)}$ to the edge $e\in E$. 
		\item A \emph{graph monomial} labeled by $\Omega$ is the couple $g=(T, \mrm{in}, \mrm{out})$ where $ \mrm{in}$ and $ \mrm{out}$ are two vertices of the test graph $T$. They are respectively called the \emph{input} and \emph{out} of $g$. 
\end{enumerate}
We denote by $\mcal G\langle \Omega \rangle$ the set of \emph{connected} graph monomial labeled by $\Omega$ and by $\mbb C\mcal G\langle \Omega \rangle$ the vector space generated by $\mcal G\langle \Omega \rangle$.
\end{definition}

The following definition shows how we can evaluate a graph polynomial in matrices to define a new matrix, generalizing the matrix product.

\begin{definition}\label{TrafficDistrCarre} Let $\mbf A = (A_\omega)_{\omega \in \Omega}$ be a collection of matrices $\mbb M_N(\mbb R)$ and let $g=(T, \mrm{in}, \mrm{out}) \in \mcal G\langle \Omega \rangle$ be a graph monomial, where $T=(V,E,\sigma)$. The \emph{evaluation of $g$ in the family $\mbf A_N$} the matrix $g(\mbf A_N)$ with entry $(i,j)\in [ {N}]^2$ 
	\eqa\label{Eq:EvalGraphPolyn}
		g(\mbf A_N) (i,j) & = & \sum_{ \substack{ \varphi : V \to [{N}]  \mrm{ \ s.t. } \\ \varphi(\mrm{out})=i, \varphi(\mrm{in})=j}}\prod_{e=(v,w) \in E} A_{\gamma(e)}\big( \varphi(w) , \varphi(v) \big).
	\qea
Assume that the entries of the matrices have finite moments of all orders. We call \emph{traffic distribution} of $\mbf A$ the map
	\eqa\label{Eq:TrafficDistribution}
		 \Phi_{\mbf A} : g \in  \mcal  G\langle \Omega \rangle  \mapsto   \esp \left [ \frac 1 N  \Tr\big[ g(\mbf A_N) \big] \right] \in \mbb R. 
	\qea
We say that $\mbf A$ \emph{converges in traffic distribution} whenever $ \Phi_{\mbf A}(g)$ converges for all $g \in  \mcal  G\langle \Omega \rangle$ as the size $N$ goes to infinity, and we say that $\mbf A$ and $\mbf B$ are \emph{traffic equivalent} if they converge the same limit.
\end{definition}

\begin{definition}\label{GraphonsDistrCarre} A collection $\mbf \Gamma$ of deterministic matrices \emph{converges in graphon distribution} whenever for any test graph $T$ labeled by $\Omega$,
	\eqa\label{Eq:GraphonDistribution}
		  \Nlim \esp \Big[ \prod_{e=(v,w) \in E} \Gamma_{\gamma(e)}\big( \Phi(w) , \Phi(v) \big) \Big] \mrm{ \ exists},
	\qea
where $\Phi: V \mapsto [N]$ is a \emph{random injective} map uniformly distributed. 
\end{definition}

\begin{example}\label{ExFonda} Let $Y = h[\{ WX\}]$, where $W$ and $X$ are square matrices and $h_n: n\mapsto x^n$. The entry-wise definitions imply that $Y =g_n(W,X)$ where $g_n$ is the graph monomial in two variables $w,x$ is as follow:  its vertex set is $\{\mrm{in}, \mrm{out}, v_1 \etc v_n\}$ and for each $i=1\etc n$, there is one edge with source $\mrm{in}$ and target $v_i$ labeled $x$, and one edge with source $v_i$ and target $\mrm{out}$ labeled $w$, see Figure \ref{F01_PWmatrix}. We call $v_1\etc v_n$ the \emph{internal vertices} of $g_n$.
\end{example}

  \begin{figure}[h]
    \begin{center}
     \includegraphics[scale=.75]{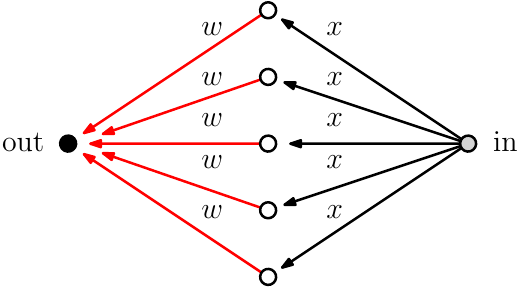}
    \end{center}
    \caption{The graph monomial $g_n$ of Pennington-Worah matrices. Note the design symmetry that the graph monomial is unchanged if we permute the internal vertices. Edges colors serves to distinguish edges labels.}
    \label{F01_PWmatrix}
  \end{figure}

The rest of the section defines how we see rectangular matrices as sub-matrices of large square matrices. Let $N_0,N_1,N_2$ three matrix size integers as in Section \ref{Sec:Intro} and set $N = N_0+ N_1 + N_2$. A matrix of $M_N(\mbb R)$ is seen as a $3\times 3$ rectangular block matrix 
\eq
	A = \left ( \begin{array}{ccc} A_{0,0} & \cdots & A_{2,2} \\ \vdots & & \vdots \\ A_{2,0} & \cdots &  A_{2,2} \end{array}\right),
\qe
where $A_{i,j} \in M_{N_i,N_j}(\mbb R)$ for each $i,j=0,1,2$. If $B\in M_{N_i,N_j}(\mbb R)$, we set $\upiota_{i,i'}(B) \in M_N(\mbb R)$ the matrix $A$ such that $A_{i,j}=B$ and $A_{i',j'}=0$ for $(i',j')\neq (i,j)$ in the block decomposition. 
Let  $\Omega$ be a label set and $\mbf s=(s_\omega,s'_\omega)_{\omega \in \Omega} \in  \{0,1,2\}^2$ be a collection of indices. We say that $ \mbf A= (A_\omega)_{\omega\in \Omega}$ is a collection of  {$\mbf s$-rectangular matrices} if $A_\omega$ is of size $N_{s_\omega}\times N_{s'_\omega}$ for all $\omega \in \Omega$ and we set $ \upiota_{\mbf s}( \mbf A) = \big( \upiota_{(s_\omega,s_\omega')}( A_\omega)\big)_{\omega\in \Omega}$. We  call traffic distribution of a collection $\mbf A$ of $\mbf s$-rectangular matrices  the traffic distribution of $\upiota_{\mbf s}(\mbf A)$. Similarly, the definition of graphon convergence extends for rectangular matrices. 

\section{Presentation of the results}

\subsection{Statement of our result}
We say in short that a random matrix is a standard Gaussian matrix if it has i.i.d. centered real Gaussian entries of variance one. For a matrix $A$ and an integer $k\geq 1$ with denote $\sqrt[\circ]{ A} := \big(\sqrt{ A(i,j) } \big)_{i,j}$ and $A^{\circ k} = \underbrace{A \circ \dots \circ A}_{k \mrm{ \ times}}$.
\begin{theorem}\label{MainTh} Let $\mbf Y = \big( Y(h) \big) _{h \in \mbb C^{(\mrm{odd})}[y]}$ be the collection of matrices indexed by odd polynomials  
	\eq
		Y(h) =  \frac {\sqrt{\psi_0}} {\sqrt{N}} h \left[ \left\{ \frac{ W X}{\sqrt{N_0}}  \right\} \right] ,  \quad \psi_0 = \frac{N_0}N,
	\qe
where $W = \Gamma_w \circ W'$ and $X = \Gamma_x \circ X'$ satisfy \ref{ProfiledModel}. We denote by $\mbf Y^{\mrm{lin}}$, $\mbf Y^{\mrm{per}}$, $\mbf Y^{\mrm{def}} $ and $\boldsymbol \epsilon$ the collections of matrices of Definition \ref{APRIORI} with $\gamma = \frac {\sqrt \psi_0}{\sqrt {N}}$ and $\gamma_0 = \frac 1 {\sqrt {N_0}}$.  

Then $\big( \mbf Y^{\mrm{lin}} ,\mbf Y^{\mrm{per}} , \mbf Y^{\mrm{def}}, \boldsymbol \epsilon \big)$ is \emph{traffic equivalent} to $\big( \mbf Y^{\lozenge\mrm{lin}} , \mbf Y^{\lozenge \mrm{per}} , \mbf Y^{\lozenge \mrm{def}}, \mbf 0 \big)$ where $\mbf Y^{\lozenge \mrm{lin}}$ and $\mbf Y^{\lozenge \mrm{per}}$ are independent, $\mbf Y^{\lozenge \mrm{def}}$ is deterministic, and are defined as follows.  Setting the bounded matrix $M_2 := \sqrt[\circ]{  N_0^{-1}  \Gamma_w^{o2} \times \Gamma_x^{o2}}$, we have 
		$$ Y^{\lozenge \mrm{lin}}(h) := \bigg( \int_{\mbb R}   h'\Big[ \Big\{ t M_2 \Big\} \Big] \frac{e^{-\frac{t^2}2}}{\sqrt{2\pi}} \mrm dt \bigg) \circ \Big( \frac{W^{\mrm{Gau}}}{\sqrt N}  \times \frac{X^{\mrm{Gau}}}{\sqrt N} \Big).$$
where $W^{\mrm{Gau}} = \Gamma_w \circ W^{\mrm{Gau}}$, $X^{\mrm{Gau}} = \Gamma_x \circ X^{\mrm{Gau}}$, and $  W^{\mrm{Gau}}$, $  X^{\mrm{Gau}}$ are independent standard Gaussian matrices. 
	Moreover, for any $m\geq 2$, we have 
	\eq
		Y^{\lozenge \mrm{per}}_m(h)  := \frac 1 {m!} \bigg( \int_{\mbb R}   h^{(m)}\Big[ \Big\{ t M_2 \Big\} \Big] \frac{e^{-\frac{t^2}2}}{\sqrt{2\pi}} \mrm dt \bigg)\circ \frac{Z_m^{\mrm{Gau}}}{\sqrt N},
	\qe
	where $Z_m^{\mrm{Gau}}, m\geq 2,$ are independent standard Gaussian matrices.
	 Finally, setting  $\Lambda_3 = N^{-1}  \Gamma_w^{o3} \times \Gamma_x^{o3}$, we have
		\eq
		Y^{\lozenge \mrm{def}}(h) & := &   		 \frac {m_w^{(3)} m_x^{(3)}}{6N}  \Lambda_3\circ \bigg( \int_{\mbb R}   h'''\Big[ \Big\{ t M_2 \Big\} \Big] \frac{e^{-\frac{t^2}2}}{\sqrt{2\pi}} \mrm dt \bigg).
	\qe
\end{theorem}

The expression of $Y^{\lozenge \mrm{lin}}(h)$ is linear in the matrix $ \frac{W^{\mrm{Gau}}}{\sqrt N}  \times \frac{X^{\mrm{Gau}}}{\sqrt N}$, but now compared to Theorem \ref{Th:EtatArt} the linear relation means an entry-wise product. Similarly, the noise part $Y^{\lozenge \mrm{def}}(h)$ is now a variance profile Gaussian random matrix. The deformation has finite rank when the profiles are constant, and its entries are $O(N^{-1})$. This deformation is known from \cite{Male2020} to do not change the limit of the limiting singular-values distribution.

\begin{example} In the context of Theorem \ref{MainTh}, assume that one of the random matrices has a constant variance profile, for instance $\Gamma_x(i,j)=1$ for all $i,j$. Therefore we see that the profile of the linear part is a rank one matrix: setting $D(h)$ the diagonal matrix whose diagonal entries coincide with those of this profile matrix, we get that $Y(h)$ is trafic equivalent to
	\eqa\label{Eq:MainTh}
	 \frac{ W^{\mrm{Gau}} X^{\mrm{Gau}}}{N} D(h) + \Theta(h) \circ\frac{ Z^{\mrm{Gau}}}{\sqrt{N}} + B(h),
	\qea
where $\Theta$ is a profile and $B(h)$ a deformation that dot not change the limiting empirical singular-values distribution.
\end{example}

\subsection{Heuristic and comments}

The contribution can be computed from Definition \ref{APRIORI} with the following approximations.
\begin{enumerate}
	\item The contribution for the linear part can be approximated jointly in traffic distribution thanks to factorizations rules
		\eq
			\gamma \gamma_0^nZ( \lambda_1^{(n)}) & \sim& \frac{ \sqrt{ \psi_0}}{\sqrt{N}} \times \frac 1 {\sqrt{N_0}^n} (W^{\mrm{Gau}}X^{\mrm{Gau}}) \circ \big( \Gamma_w^{o2} \times \Gamma_x^{o2} \big)^{\circ\frac{ (n-1)}2}\\
			& = & \Big( \frac {W^{\mrm{Gau}}} {\sqrt N}  \frac {X^{\mrm{Gau}}} {\sqrt N}  \Big) \circ \Big( \frac 1 {N_0} \Gamma_w^{o2} \times \Gamma_x^{o2} \Big)^{\circ\frac{ (n-1)}2}
		\qe
which indeed gives the expected equivalent
	$$Z^{\mrm{lin}}(h) \sim   \sum_{n\geq 1} a_n \Big(  \int_{\mbb R}  h_n'[\{ t M_2 \}] \mrm d \omega (t) \Big) \circ \Big( \frac {W^{\mrm{Gau}}} {\sqrt N}  \frac {X^{\mrm{Gau}}} {\sqrt N}  \Big) = Y^{\lozenge \mrm{lin}}(h).$$
	\item For the deformation term, we have asymptotically
	$$Z( \lambda_3^{(n)}) \sim \frac {m_w^{(3)} m_x^{(3)}}{6N}  \Lambda_3\circ  M_2^{\circ(n-1)},$$
 so again the definition of $Y^{\mrm{def}}$ results in an entry-wise evaluation formula for $Y^{\lozenge\mrm{def}}$.
	\item Finally, for each $m\geq 2$,
		\eq
			\gamma \gamma_0^nZ( \lambda_1^{(m,n)}) & \sim &   \frac{ \sqrt{ \psi_0}}{\sqrt{N}} \times \frac 1 {\sqrt{N_0}^n} Z( \lambda_1^{(m,m)})\circ  \big( \Gamma_w^{o2} \times \Gamma_x^{o2} \big)^{\circ\frac{ (n-m)}2}\\
			& = &\sqrt{N_0}^{m-1} \times \frac {\psi_0} {N_0^{2m}} Z(\lambda_1^{(m,m)}) \circ \big( \frac 1 {N_0} \Gamma_w^{o2} \times \Gamma_x^{o2} \big)^{\circ\frac{ (n-m)}2}
		\qe
	Denoting $Z = \frac W {\sqrt N_0}  \frac X {\sqrt N_0} = Z_1 \times Z_2$, we have for instance with $m=2$  
		$$\frac 1 {N_0^{2m}} Z(\lambda_1^{(2,2)})   =  \Big( Z \circ Z - Z_1^{\circ 2} \circ Z_2^{\circ 2} \Big),$$
which is known to converges to zero in traffic distribution. Hence the perturbative term can be interpreted as fluctuations around this limit. The computations show that $\sqrt{N_0}^{m-1} \times \frac {\psi_0} {N_0^{2m}} Y(\lambda_1^{(m,m)})$ is traffic equivalent to 
	$$ \big( \frac 1 {N_0} \Gamma_w^{o2} \times \Gamma_x^{o2} \big)^{\circ\frac{ m}2} \circ \Big(  \int_{\mbb R} h_m \mrm d \omega (t) \Big) 
  \circ \frac{Z_m^{\mrm{Gau}}}{\sqrt N},$$
 which gives the expected equivalent.
\end{enumerate}

In the next section we state the generalization of the above theorem in terms of traffic distribution, for which the analog of Equation \eqref{Eq:MainTh} requires a third additional term.

\subsection{P\'ech\'e perspective and free probability}

To emphasis the strength of the equivalent method, recall that free probability gives {\bf a robust and systematic way to compute eigenvalues distributions}. In the context of Theorem \ref{Th:EtatArt}, since the spectral equivalent is a polynomial in independent matrices, the consistent system of equations for the its Stieltjes transform for the distribution of $YY^t$, which are crucial for numerical applications, can be derived thanks to Dan Virgil Voiculescu's free probability theory \cite{Voiculescu1991}. In particular, a linearization trick developed by \cite{BMS17} allows to reformulate this system as the fundamental \emph{subordination property}. This approach is quite robust since for more structured spaces the subordination property holds ''with a twist``, formulated thanks to the powerful concept of \emph{amalgamation} (a non commutative analogue of conditional probability). The method has its limitations, the systems of equations for Stieltjes transforms obtained by free probability may be quite complicated, but it provides explicit algorithms. 
	
	Moreover, {\bf the method extends for strong equivalents.} The strong convergence of independent real Gaussian matrices is known by a result of Catherine Donati-Martin and Mireille Capitaine \cite{CD07} and the latter author show that finite rank deformations can be described in a simple way thanks to the subordination property \cite{Cap13}. Therefore the computation of the outliers for the strong equivalent can be made with the same robust method as for the computation of the eigenvalues distribution.

Beyond free probability, our approach use a specialization of free probability called traffic probability \cite{Male2020}, which is motivated by the \emph{distributional invariance} of Pennington-Worah matrices. Let us denote by $\mcal O = \mcal O_N$ the set  of size $N$ orthogonal real matrices  and by $\mrm{Perm} = \mrm{Perm}_N$ of size $N$ permutation matrices are two compact subgroups $\mrm M_N(\mbb R)$. Let us describe the two following situations, given a collection of random matrices $\mbf A=(A_j)_{j\in J}$ of size $N_1\times N_2$, where implicitly $N_1$ and $N_2$ goes to infinity such that $\frac{N_i}N \limN \psi_i>0$ for an auxiliary parameter $N\longrightarrow \infty$.
\begin{enumerate}
	\item The collection of matrices is \emph{bi-unitarily invariant} in the sense that $\mbf A$ has the same law as $U\mbf AV=(UA_jV)_{j\in J}$ for all orthogonal matrices $U, V$ of corresponding size. {\bf Then we can use Florent Benaych-Georges's rectangular free probability theory} \cite{BG09}, which correspond to the setting of amalgamation over a space of finite dimensional matrices. In the context of Theorem \ref{Th:EtatArt}, we have  \emph{convergence in joint non-commutative distribution }
	\eq
		\big(Y^{\square\mrm{lin}}(h), Y^{\square\mrm{per}}(h) \Big)  \limN  \big(  y^{\mrm{lin}}(h), y^{\mrm{per}}(h) \big)
	\qe
in the sense that for any polynomial $P$ in two non commutative indeterminates consistent with relative matrix sizes, the limit 
	$$\esp\Big[ \frac 1 N P\big( Y^{\square\mrm{lin}}(h), Y^{\square\mrm{per}}(h)\big) \Big] \limN \Phi(P),$$
exists. The map $\Phi$ is a linear form on the space of non commutative polynomials. The framework of non-commutative probability allows to think these indeterminates as random variables by ''mimicking`` classical probability in a non-commutative way: the limit of matrices are called \emph{non-commutative random variables} and their are determined by their \emph{non-commutative distribution} $\Phi$.  Non-commutative random variables  are understood to be in ''generic position`` when there are \emph{freely independent} (or simply free), which means that their distribution satisfies a canonical expression in terms of marginal distribution. The bi-orthogonal invariance of $Y^{\square\mrm{lin}}(h)$ and $Y^{\square\mrm{per}}(h)$ explains that their limit are free.	
	 
	  	\item The collection of matrices is \emph{bi-permutation invariant} in the sense that $\mbf A$ has the same law as $U\mbf AV=(UA_jV)_{j\in J}$ for all permutation matrices $U, V$.  {\bf Then we can use Greg Zitelli rectangular traffic probability theory} \cite{Zit24}. Note that the collection of matrices $Y(h)$ itself is bi-permutation invariant. For a collection $\mcal H$ of functions $h$ and under assumptions stated next section, we have  \emph{convergence in traffic distribution }
	\eq
		\mbf Y = \big (Y(h) \big)_{h\in \mcal H} \limN \mbf y = \big (y(h) \big)_{h\in \mcal H},
	\qe
in the sense the for any function $g$ in a class of functions defined below and called the split \emph{graph-polynomials}, the limit 
	$$\esp\Big[ \frac 1 N g\big(  \mbf  Y \big) \Big] \limN \Phi(g),$$
exists. As graph polynomials generalize non commutative polynomials, the traffic distribution extends the non-commutative distribution. The interest is that the traffic distribution gives much more information, capturing certain finite rank deformations and distinguishing easily bi-unitarily invariant matrices. Traffic distributions comes with a notion of \emph{traffic-independence} which ''encompasses`` \cite{CDM24} three notions of probability:  classical probability referred as \emph{tensor-probability}, the free probability, and the \emph{Boolean-probability} which is another non-commutative non of independence. 

\end{enumerate}

\section{Asymptotic traffic distribution}
\subsection{Definitions and Notations}

Our method is the same as in \cite{BenPec}, using the language of rectangular traffic probability. In this first subsection, we present a slight reformulation of the traffic distribution of matrices relating the square and rectangular cases. We also introduce an important transform, the injective traffic distribution, and the ingredient to link these two notions. We recall that for $i,i'\in \{0,1,2\}$, a $N_i\times N_{i'}$ matrix $A$ is canonically associated a square matrix $\upiota_{i,i'}(A)$ of size $N=N_0+N_1+N_2$ by completing $A$ with zeros.

Let $\mbf s=(s_\omega,s'_\omega)_{\omega \in \Omega} \in  \{0,1,2\}^2$ be a collection of matrix size indices. A test graph $T=(V,E, \gamma)$ labeled in $\Omega$ is \emph{split} (with respect to $\mbf s$) if there is a partition $V = V_0 \sqcup V_1 \sqcup V_2$ such it links edges whose label matches the size indices: for any $e = (v,w) \in E$, $v\in V_p$ and $w\in V_q$ implies $(s_{\gamma(e)}, s'{\gamma(e)}) = (q,p)$. A  map $\phi:V\to [N]$ is split whenever it sends $V_i$ to the range isomorphic to $[N_i]$ in $[N]$.

\begin{definition}\label{TrafficDistrComb}
 \begin{itemize}
	\item Let $\mbf s=(s_\omega,s'_\omega)_{\omega \in \Omega} \in  \{0,1,2\}^2$ and $\mbf A = (A_\omega)_{\omega \in \Omega}$ be a family of $\mbf s$-rectangular matrices, and denote $\tilde  A_\omega = \upiota_{(s_{\omega}, s'_{\omega})} (A_\omega )$ for all ${\omega \in \Omega}$. Let $T=(V,E,\gamma)$ be a split test graph labeled by $\Omega$. We call \emph{combinatorial trace of $T$ in the family $\mbf A_N$} the complex number
	\eqa\label{Eq:EvalTestGraph}
		 {\Tr \big[ T(\mbf A) \big] } & = & \sum_{ \substack{ \varphi : V \to [ {N}] }}\prod_{e \in E}  \tilde A_{\gamma(e)}\big( \varphi(e)\big),
	\qea
where for $e=(v,w)$ we denote $ \varphi(e)=\big( \varphi(w) , \varphi(v)\big)$. The sum can be restricted to split maps since otherwise the summand vanishes. We call \emph{injective trace of $T$ in the family $\mbf A_N$}, and denote $\Tr^0 \big[ T(\mbf A) \big] $, the quantity defined as above when the sum over $\varphi$ is restricted to injective split maps. We denote
	\eq
		\tau_N\big[T(\mbf A)\big] := \esp\left[ \frac 1 N \Tr \big[T(\mbf A)\big]  \right],  \quad \tau^0_N\big[T(\mbf A)\big] := \esp\left[ \frac 1 N \Tr^0 \big[T(\mbf A)\big]  \right].
	\qe
	\item For any test graph $T=(V,E,\gamma)$ and any partition $\pi$ of its vertex set $V$, we denote by $T^\pi=(V^\pi,E^\pi,\gamma^\pi)$ the test graph obtained by identifying vertices in a same block of the partition $\pi$: $V^\pi = \pi$ and each edge $e=(v,w)$ in $E$ induces an edge with source and target the block of $\pi$ that contains $v$ and $w$ respectively. A partition $\pi \in \mcal P(V)$ of $V$ is split whenever $\pi$ does not identify vertices of different $  V_i$'s, namely $\mcal V^\pi =   V_0^\pi \sqcup   V_1^\pi \sqcup   V_2^\pi$. 
\end{itemize}
\end{definition}

There is an abuse of notation using the term "trace" in the theorem above since the combinatorial and the injective traces are not defined on matrix spaces. The combinatorial trace is related to the usual trace of matrices as follows: for any family of matrices $\mbf A$ and any graph monomial $g = (T, \mrm{in}, \mrm{out})$ one has $\Tr \, g(\mbf A_N) = \Tr \big[  T'(A) \big]$, where $  T'$ is obtained from $T$ by identifying $\mrm{in}$ and $\mrm{out}$. On the other hand, for any test graph $T$, we have the identity
	$$\Tr \big[ T(\mbf A) \big] = \sum_{\pi \in \mcal P(V)} \Tr^0 \big[ T^\pi(\mbf A) \big],$$
where $\mcal P(V)$ is the set of partitions of $V$. The sum can be restricted to the split partitions $\pi$ of $V$. 
\section{Proof of the main theorem}

\subsection{Traffic method for profiled Pennington-Worah matrices}\label{PWModel}

Let us set $\mbf Y = (Y_h)_{h \in \mbb R[y]}$ the family of matrices
\eqa\label{Yk}
    Y(h) =  \frac {\sqrt{\psi_0}} {\sqrt{N}} h \left[ \left\{ \frac{ W X}{\sqrt{N_0}}  \right\} \right] ,  \quad \psi_0 = \frac{N_0}N,
\qea
where $W$ and $X$ are as in Theorem \ref{MainTh}, and $h[\{ \, \cdot \,\}]$ stands for the entry-wise evaluation.  Since \eqref{Yk} is a linear expression in $h$, to prove the convergence in traffic distribution of $\mbf Y$ it suffices to prove the convergence of $\esp[\frac 1 N \Tr \big[T(\mbf Y)\big]$ for test graphs indexed by a basis of $\mbb R[y]$. 

We call \emph{reference graphs} the split test graphs $T$ in variables standing for the matrices of $\mbf Y$ labeled by monomials. More precisely, we denote $T = (V,E, \mbf n)$ where the index map $ \mbf n : E \to \mbb N$ tells that an edge $e\in E$ is associated to the matrix $Y_{\mbf n(e)} := Y(h_{\mbf n(e)})$ where $h_n:x\mapsto x^{  n}$. We have a partition $V=V_1\sqcup V_2$ so that each edge of $T$ has its source in $V_2$ and its target in $V_1$. Later at the end of the proof we shall use test graph labeled by Hermite polynomials.

We give an expression of $\esp[\frac 1 N \Tr \big[T(\mbf Y)\big]$ for any reference graph $T$, we use first the so-called substitution property of graph monomials and test graphs (see \cite[Definition 1.7]{Male2020}) to re-write \eqref{Yk} in graph language. The proof of the main theorem follows as we can clearly identify the contribution of each components in the Pennington-Worah matrix decomposition.

In the definition below, we use the terminology of \cite{BenPec}.

\begin{definition}\label{TrondGraph} Given $T=(V,E,\mbf n)$ a reference test graph, we denote by $\mcal T_T= (\mcal V, \mcal E, \delta)$ the (auxiliary) test graph in two variables $w$ and $x$ obtained from $T$ as follows. It is obtained by considering the vertex set of the reference graph (that we then call the \emph{reference vertices}) and adding for each edge $e\in E$ a collection $\mcal N(e) = (\mcal V_e, \mcal E_e)$ of vertices and edges, that is called the \emph{niche} of $e$, defined as follow. There are $\mbf n(e)$ vertices in the niche of $e$, different from the reference vertices and the vertices of other niches. They are called the \emph{internal} vertices. Moreover, the niche of $e$ has $2\mbf n(e)$ edges, as each internal vertex is both 
\begin{itemize}
    \item the source of an edge labeled $w$ whose target is the target of $e$ in $T$,
    \item the target of an edge denoted labeled $x$ whose source is the source of $e$ in $T$.  
\end{itemize}
Two adjacent edges in $\mcal T$ that share the same internal vertex are say to be \emph{companion} each other.
\end{definition}

When there is no ambiguity about the reference graph $T$, we write in short $\mcal T:= \mcal T_T$. The definitions imply that we have
\eqa
	\nonumber \tau_N \big[T(\mbf Y)\big] &=&\Big(\frac{\psi_0}N\Big)^{\frac{|E|}{2}} N_0^{-\sum\limits_{e\in E} \frac{\mbf n(e)}{2}} \tau_N  \big[\mcal T(W,X)\big]\\
	& =& \psi_0^{-\sum\limits_{e\in E} \frac{\mbf n(e)-1}{2}} N^{-\frac{|E|}{2}-\sum\limits_{e\in E} \frac{\mbf n(e)}{2}} \sum_{\pi \in \mathcal{P}(\mathcal{V})} \tau_N^0[\mathcal{T}^\pi(W,X)].\label{Eq:Substitution}
\qea

Recall that Hypothesis \ref{ProfiledModel} tells that $W = \Gamma_w \circ W'$ and $X = \Gamma_X \circ X'$, where $W',X'$ are independent i.i.d. matrices and $\Gamma_w, \Gamma_x$ are bounded matrices. Next we explain how we can factorize the contribution of profiles and work on the traffic distribution of $W'$ and $X'$. 

\begin{definition}\label{Def:delta} For any family of $\mbf s$-rectangular random matrices $\mbf A=(A_\omega)_{\omega \in \Omega}$ and any split test graph $ \mathfrak  T=(\mathfrak V,\mathfrak E,\gamma)$ labeled in $\Omega$, with the same notations as in Definition \ref{TrafficDistrComb}, we set
	\eq
		\delta^0\big[\mathfrak T(\mbf A) \big] & = & \esp\left[ \prod_{e \in\mathfrak  E} \tilde A_{\gamma(e)}\big( \Phi(e)  \big) \right],
	\qe
where $\Phi:\mathfrak V\to [N]$ is a injective split map uniformly distributed at random independently of $\mbf A$. 
\end{definition}
Let us denote $(m)_n= \frac{m!}{(m-n)!}$ (known as the falling factorial notation), which is the number of injective maps from $[n]$ to $[m]$. From the definition of $\tau^0_N$, note that we have
	\eqa
		\tau_N^0\big[\mathfrak T(\mbf A) \big]  &=& \frac 1 N \mrm{Card}\{ \pi :\mathfrak V\to [N] \mrm{\ split, \ injective}\} \delta^0\big[\mathfrak T(\mbf A) \big]  \nonumber\\
		 & = & \frac 1 N  (N_0)_{|\mathfrak V_0|}(N_1)_{|\mathfrak V_1|}(N_2)_{|\mathfrak V_2|} \delta^0\big[\mathfrak T(\mbf A) \big]  \label{DeltaTrickZero}.
	\qea
 On the other hand, for $\mcal T=\mcal T_T$ as in \eqref{Eq:Substitution}, we denote by  $\mcal E^\pi_w$ and $\mcal E^\pi_x$ the set of edges of $\mcal T$ labeled $w$ and $x$ respectively. The independence of $W'$ and $X'$ implies 
\eq
	\tau_N^0\big[\mathcal{T}^\pi(W,X)\big]\Big]  
	& = & \frac 1 N  \sum_{ \substack{ \phi:\mcal V^\pi\to [N] \\ \mrm{injective}}} \Bigg( \prod_{ e \in \mcal E^\pi_w}  \Gamma_w\big(\phi(e) \big)  \times  \prod_{ e \in \mcal E^\pi_x}     \Gamma_x\big(\phi(e) \big)\nonumber \\
	& & \times  \esp \bigg[ \prod_{ e \in \mcal E^\pi_w}  W'\big(\phi(e) \big)  \bigg] \times   \esp \bigg[ \prod_{ e \in \mcal E^\pi_x}  X'\big(\phi(e) \big) \bigg]\Bigg).
\qe
Since the matrices $W'$ and $X'$ have i.i.d. entries, the values of the expectation does not change if we change the split partition $\pi$ into another arbitrary split partition $\Phi$. This is still true if $\Phi$ is uniformly distributed at random independently of the matrices. Let us define the test graph $\mathcal{T}_w = (\mathcal{V}_w,\mathcal{E}_w)$ in one variable $w$ obtained from $\mcal T$ by removing the edges labeled $x$ and the vertices of $  V_2$, and let $\mcal T_x$ be defined similarly. Hence we have
\eqa\label{DeltaTrick}
	\lefteqn{\tau_N^0\big[\mathcal{T}^\pi(W,X)\big]\Big]   }  \\
	& = &\frac 1 N  \sum_{ \substack{ \phi:\mcal V^\pi\to [N] \\ \mrm{injective}}} \Bigg( \prod_{ e \in \mcal E_w}  \Gamma_w\big(\phi(e) \big)     \prod_{ e \in \mcal E_x}     \Gamma_x\big(\phi(e) \big) \Bigg)   \delta^0[\mcal T^\pi_w(W')]  \delta^0[\mcal T^\pi_x(X')] \nonumber \\
	& = &  \frac 1 N  (N_0)_{V_0}(N_1)_{V_1}(N_2)_{V_2}  \delta^0\big[ \mcal T^\pi(\Gamma_w, \Gamma_x) \big]    	 \delta^0[\mcal T^\pi_w(W')]  \delta^0[\mcal T^\pi_x(X')],\nonumber 
\qea
Note that since the entries of $\Gamma_w$ and $ \Gamma_x$ are bounded the term $\delta^0\big[ \mcal T^\pi(\Gamma_w, \Gamma_x) \big] $ is bounded. Moreover, $\delta^0[\mcal T^\pi_w(W')]  \delta^0[\mcal T^\pi_x(X')]=\delta^0[\mcal T^\pi(W',X')] $ is independent of $N$. Finally, recalling that $\mcal V^\pi = V_0^\pi \sqcup V_1^\pi \sqcup V_2^\pi$ since $\pi$ is split, we have 
	$$(N_0)_{V_0}(N_1)_{V_1}(N_2)_{V_2} = N^{|\mathcal{V}^\pi|} \psi_0^{|V_0^\pi|} \psi_1^{|V_1^\pi|}\psi_2^{|V_2^\pi|} \big( 1 + o(1) \big).$$
 Therefore, by \eqref{Eq:Substitution} and \eqref{DeltaTrick}, setting 
\eqa
	\eta (\pi) &=& -1 +  |\mathcal{V}^\pi| -\frac{|E|}{2}-\sum\limits_{e\in E}  \frac{\mbf n(e)}{2},\nonumber\\
	\Psi^\pi & =&  \psi_0^{|V_0^\pi| -\sum\limits_{e\in E} \frac{\mbf n(e)-1}{2}}  \psi_2^{|V_2^\pi|}  \psi_1^{|V_1^\pi|}, \quad \psi_i = \frac{N_i}{ {N}},\label{DefPsi}
\qea
  the following expression is valid for any reference graph $T$
	\eqa\label{TauToC}
	 \lefteqn{\tau_N \big[ T( Y) \big]  }\\
		& =&    \sum_{ \pi \in \mathcal{P}(\mathcal{V}) }  N^{\eta(\pi)} \Psi^\pi\ \delta^0\big[\mcal T^\pi(\Gamma_w, \Gamma_x)\big]\delta^0[\mcal T^\pi_w(W')]  \delta^0[\mcal T^\pi_x(X')]\big( 1 + o(1) \big) \nonumber
\qea

\subsection{Convergence and support of the limit}\label{Sec:ConvEtSupp}

In a first subsection, we prove that for a reference graph $T$ as in the previous section and for any $\pi\in \mcal P(\mcal V)$ such that $\delta^0[\mcal T^\pi_w(W')] \times  \delta^0[\mcal T^\pi_x(X')] = \delta^0[\mcal T^\pi(W',X')]\neq 0$, we always have $ \eta (\pi) \leq 0$. The validity of this claim implies the convergence in traffic distribution of $\mbf Y$ since therefore in the r.h.s. of Formula \eqref{TauToC} all terms are bounded.  In a second subsection, we state an intermediate result that allows us to identify the partitions such that $\delta^0[\mcal T^\pi(W',X')]\neq 0$ and $ \eta (\pi) =0$.

\subsubsection{Proof of the convergence of the traffic distribution of $\mbf Y$}

In this subsection, we fix a split reference graph $T$. Let $\pi \in \mcal P(\mcal V)$ be a split partition and denote by $\rho(\pi)$ the restriction of $\pi$ to the vertex set of $T$. In the following expression 
	\eqa\label{TransfEta1}
		\eta(\pi) = |\mathcal{V}^\pi|-1-\frac{|E|}{2}-\sum_{e\in E} \frac{\mbf n(e)}{2},
	\qea
 only $|\mcal V^\pi|$ depends on the partition $\pi$, so $\eta(\pi)$ is large when the number of vertices of $\mcal T^{\pi}$ is large. On the other hand, by independence of the matrices and their entries, we must have edges identifications in order to have $\delta^0[\mcal T^\pi(W',X')]\neq 0$. We shall prove that the competition between these two constraints results in terms of the right order.

The strategy consists in translating the condition $\eta(\pi)$ into concrete conditions on two other intermediary graphs. First recall that $\mathcal{T}_w$ is the graph labeled in $w$ obtained from $\mcal T$ by removing the edges labeled $x$ and the vertices of $  V_2$. It may be a disconnected graph, so we denote by $c_\pi\geq 1$ its number of connected components. Moreover, the number of edges labeled $w$ in $\mcal T^\pi$ is $|\mathcal{E}_w^\pi| = |\mathcal{E}_w| = \sum_{e\in E} \mbf n(e)$. Finally, since the partition $\pi$ is split, the number of vertices of this graph is $|\mathcal{V}_w^\pi| = |  V_1^\pi|+|  V_0^\pi|$. Hence we have the identity
	 \eqa\label{TransfEta2}
	 	\eta_1(\pi) := |\mathcal{V}_w^\pi|-c_\pi-\frac{| {\mathcal{E}}^\pi_w|}2 = |  V_1^\pi|+|  V_0^\pi| - c_\pi - \sum_{e\in E}  \frac{\mbf n(e)}2.
	\qea
We put the emphasis on this formula since $\eta_1$ replaces $\eta$ the computation of the traffic distribution of simpler models, such as independent large Wigner matrices in \cite[Chapter 3]{Male2020}. We recall the following definitions.

\begin{definition}
Let $\mathfrak G = (\mathfrak V, \mathfrak E)$ be a graph. 
\begin{enumerate}
	\item The graph $\mathfrak G$ is a forest whenever the removal of any edge of $\mathfrak G$ always increases its number of connected components.
	\item The \emph{skeleton graph} $\bar{\mathfrak G} = (\bar{\mathfrak V},\bar{\mathfrak E})$ of $\mathfrak G$ is obtained by identifying the edges of $\mathfrak G$ with the same endpoints, hence forgetting the multiplicity of the edges.
\end{enumerate}
\end{definition}

We also recall the following result (for a proof, see \cite[Lemma 2.13]{Male2020}). 
 
  \begin{lemma}\label{eq_graph}
  For any graph $\mathfrak G= (\mathfrak V,\mathfrak E)$ with $c$ connected components, 
  	$$  |\mathfrak V|-c-|\mathfrak E| \leq 0,$$
  with equality if and only if the graph is a forest.
  \end{lemma}

Therefore we re-rewrite \eqref{TransfEta2} as 
	$$\eta_1(\pi) =  \Big( |\mathcal{V}_w^\pi|-c_\pi - | \bar{{\mathcal{E}}}^\pi_w| \Big) + \Big( | \bar{{\mathcal{E}}}^\pi_w| -\frac{| {\mathcal{E}}^\pi_w|}2 \Big),$$
where ${\mathcal{E}}^\pi_w$ is the set of edges of the skeleton graph of $\mcal T^\pi_w$. The first term in the r.h.s. is non-negative by Lemma \ref{eq_graph}. On the other hand, for any $i\geq 1$ denote by $\mathcal{E}^\pi_{w, i} \subset \mathcal{E}^\pi_w$ the set of edges of multiplicity equal to $i$ in $\mathcal{T}_x^\pi$. Then we have
	\eqa\label{Eq:SkelCard}
		\Big(|\bar{\mathcal{E}}^\pi_w|-\frac{|\mathcal{E}_w^\pi|}{2}\Big) = \frac {|\mathcal{E}^\pi_{w, 1}|}2 - \sum_{i\geq 3} \frac{i-2}2  |\mathcal{E}^\pi_{w, i}|.
	\qea
If $ \delta^0 \big[ \mcal T_w^\pi(W')\big]=0$, then $T^\pi_w$ has no edge of multiplicity 1 so
$(|\bar{\mathcal{E}}^\pi_w|-\frac{|\mathcal{E}^\pi|}{2}) \leq 0$ with equality if and only if the edges labeled $w$ are of multiplicity 2 in $ \mcal T_w^\pi$. 

As we cannot find a way to apply the same reasoning $\eta$ defined in \eqref{TransfEta1}, we introduce a second graph.  We consider $T^{\tilde \rho(\pi)}$ the quotient of $T$ by the split partition $\tilde \rho(\pi)$  such that
\begin{itemize}
\item $\forall i,i' \in V_1$,   $i\sim_{\tilde \rho(\pi)} i' $ whenever $i$ and $i'$ belong to the same connected component of the quotient of $\mathcal{T}_w$ by the restriction of $\pi$,
\item $\rho(\pi)$ and $\tilde \rho(\pi)$ coincide on $V_2$ ($\forall j,j' \in V_2$, $j\sim_{\tilde \rho(\pi)} j' \Leftrightarrow j \sim_\pi j'$).

\end{itemize}

Note that the partition $\rho(\pi)$, which is induced by $\pi$ by restriction, is finer than $\tilde \rho(\pi)$ in the sense that each block of $\rho(\pi)$ belong to a block of $\tilde \rho(\pi)$. Moreover $T^{\tilde \rho(\pi)}$ is a connected graph since $T$ is connected. Its number of edges is the same as for $T$, namely $|E^{\tilde \rho(\pi)}|=|E|$. It number of vertices is $|  V^{\tilde \rho(\pi)}| = c_\pi + |  V_2^\pi|$. Therefore we have
	\eqa\label{TransfEta3}
 		\eta_2(\pi)  &:= & |\mcal V^{\tilde \rho(\pi)}|-1-\frac{|E^{\tilde \rho(\pi)}|}2 = c_\pi + |  V_2^\pi|-1 -   \frac{|E|}2,
	\qea
and Formulae \eqref{TransfEta1}, \eqref{TransfEta2}  imply that $\eta(\pi) = \eta_1(\pi) + \eta_2(\pi)$. With   $|\bar{E}^{\tilde \rho(\pi)}|$ standing for the number of edges of the skeleton of $T^{\tilde \rho(\pi)}$, we write as before
 \eqa
 \eta_2 (\pi) 
	& =&  \Big(|\bar{\mathcal{E}}^\pi_w|-\frac{|\mathcal{E}_w^\pi|}{2}\Big)+\Big(|\bar{ {E}}^{\tilde \rho(\pi)}|-\frac{| {E}^{\tilde \rho(\pi)}|}{2}\Big),\label{Eq:4termExp}
 \qea
The first term is non-negative when $\delta^0[\mcal T^\pi(W',X')]\neq 0$. By a computation similar to \eqref{Eq:SkelCard}, we see that last term $|\bar{E}^{\tilde \rho(\pi)}|-\frac{|E^{\tilde \rho(\pi)}|}{2}$ is bounded by half the number of edges of multiplicity one in $T^{\tilde \rho(\pi)}$, which may be positive. We must therefore show that when this quantity is positive, another quantity compensates it. The following definition clarifies phrasing that we use through the proof.

\begin{definition}\label{def:Extra_Niche}
Let $\pi$ be a partition of  $\mathcal{V}$. We say that a group of edges of $\mcal T^{\pi}$ are \emph{identified} by $\pi$ if their target vertices belong to a same block of $\pi$, as long as their sources. An \emph{extra-niche} (respectively \emph{intra-niche}) \emph{identification} is an identification of internal vertices or edges of $\mcal T^{\pi}$ from different niches (respectively the same niche). Two edges $e$ and $e'$ of $T$ are \emph{niche neighbors} ($w$-niche neighbors or $x$-niche neighbors)  whenever they have in their niche edges forming an extra-niche identification (and these edges are labeled $w$ or $x$ respectively).
\end{definition}

The following fact is used several times.

\begin{lemma}\label{Lem:ExtraXbound}Assume that $T$ and $\pi \in \mcal P(\mcal V)$ are split and $\delta^0[\mcal T^\pi(W',X')]\neq 0$. If two edges $e$ and $e'$ of $T$ are $x$-neighbors, then they form a group of edges of multiplicity at least 2 in $T^{\tilde \rho(\pi)}$.
\end{lemma}

\begin{proof} Denote by $e_x$ and $e'_x$ two edges of $e$ and $e'$ respectively forming an extra-niche identification. The targets of $e_x$ and $e'_x$ coincide in $\mcal T^\pi$, so the targets of $e$ and $e'$ belong to the the same connected component in $\mcal T^\pi$. Hence by definition of $\tilde \rho$, $e$ and $e'$ form a group of multiplicity at least 2 in $T^{\tilde \rho(\pi)}$. \end{proof}

\begin{lemma}\label{simpedge} Assume that $T$ and $\pi \in \mcal P(\mcal V)$ are split and $\delta^0[\mcal T^\pi(W',X')]\neq 0$. 
Let $e\in  E$ be a simple edge in $T^{\tilde \rho(\pi)}$. Then $e$ has no $x$-neighbor, and  the internal vertex of its niche form intra-niche identifications.
\end{lemma}

\begin{proof}
A edge $e\in E$ that is simple in $T^{\tilde \rho(\pi)}$ has no $x$-neighbor since otherwise it would contradict Lemma \ref{Lem:ExtraXbound}.  Moreover, by the centering and the independent of the matrices $W', X'$ and their entries, the condition $\delta^0[\mcal T^\pi(W',X')] \neq 0$ implies that the edges labeled $x$ in the niche of $e$ must be identified somewhere, so necessary they form intra-niche identifications. 
\end{proof}

We can now prove that $\eta(\pi)\leq 0$. Let us denote by $l^\pi$ the number of simple edges of $T^{\tilde \rho(\pi)}$ and by ${l'}^\pi$ the number of edge of multiplicity at least 3 in $\mathcal{T}_w^\pi$. Since $\mcal T^\pi_w$ has no simple edge when $\delta^0[\mcal T^\pi(W',X')]\neq 0$, then \eqref{Eq:SkelCard}  shows that
$|\bar{\mathcal{E}}_w^\pi| -  \frac{|\mathcal{E}_w^\pi|}{2}  \leq   -\frac{{l'}^\pi}{2}$. On the other hand, since each niche contains an odd number of internal vertices, Lemma \ref{simpedge} implies that the niches corresponding to simple edges of $T^{\tilde \rho(\pi)}$ must have an edge of multiplicity at least $3$, and so we have $
l^\pi\leq {l'}^\pi$. Moreover, from the definitions of $l^\pi$ we get 
 $| \bar{E}^{\tilde \rho(\pi)}| - \frac{|E^{\tilde \rho(\pi)}|}{2} \leq  \frac{l^\pi}{2}$. 
Formula \eqref{Eq:4termExp} and the above arguments imply that $\eta (\pi) \leq    \frac{l^\pi-{l'^\pi}}{2} \leq  0$ whenever $\delta^0[\mcal T^\pi(W',X')] \neq 0$ and so $\eta (\pi) \leq 0$. This proves the convergence $ C_N(\mcal T^\pi) $, and so the convergence in traffic distribution of $\mbf Y$ as explain in the presentation of Section \ref{Sec:ConvEtSupp}.

\subsubsection{Support of the traffic distribution}\label{Sub:Support}

The condition $\eta(\pi)=0$ is hence equivalent to the following four conditions for the intermediary graphs $\mcal T_w^\pi$ and $T^{\tilde \rho(\pi)}$ introduced before:
\eqa\label{TreeCal}
-1+c_\pi+|V_2^\pi|-|\bar{E}^{\tilde \rho(\pi)}| &=& 0, \\
 \label{DbCal}
 | \bar{E}^{\tilde \rho(\pi)}| - \frac{|E^{\tilde \rho(\pi)}|}{2} & = & \frac{l^\pi}{2},\\
 \label{Db}
 |\bar{\mathcal{E}}_w^\pi| -  \frac{|\mathcal{E}_w^\pi|}{2} & = & -\frac{l^\pi}{2}\\
\label{Forest}
|\mathcal{V}_w^\pi|-c_\pi-|\bar{\mathcal{E}}^\pi_w| & = & 0,
\qea
with $l^\pi$ the number of simple edges of $T^{\tilde \rho(\pi)}$. 

\begin{definition}\label{DefCactus} \begin{enumerate} 
\item A \emph{simple cycle} in a graph is a sequence of pairwise-distinct vertices $v_1\etc v_k$, such that $v_i$ and $v_{i+1}$ are adjacent, with indices modulo $k$ (there is no restriction on the directions of the edges).
\item A \emph{cut edge} of a graph is an edge whose removal increases the number of connected components. The set of cut edges of a graph $G$ is denoted $\mcal C_1(G)$.
\item A \emph{cactus} (respectively a \emph{pseudo-cactus}) is a graph such that each edge belongs to exactly (respectively at most) one simple cycle. 
\item A \emph{strong component} of a pseudo-cactus $G$ is whether a simple cycle or a cut edge of $G$, and we denote by $\mcal {SC}(G)$ the set of strong components of $G$, see an example Figure \ref{F03_PCactus}. We call \emph{cut vertex} a vertex of a graph that belong to several strong components.
\end{enumerate}
\end{definition}

  \begin{figure}[h]
    \begin{center}
     \includegraphics[scale=.75]{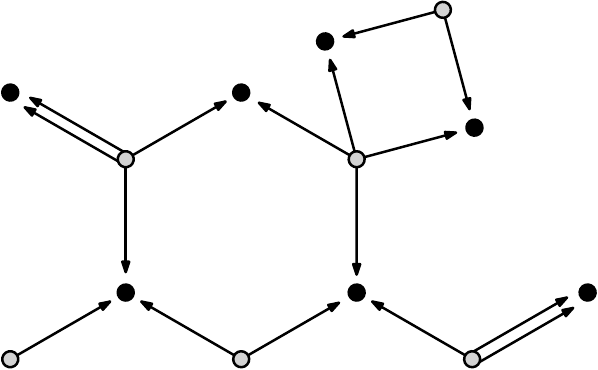}
    \end{center}
    \caption{A pseudo-cactus with two cut edges, two length 2 simple cycles, one length 4 and one length 6 cycles.}
    \label{F03_PCactus}
  \end{figure}

\begin{proposition}\label{Prop:Support} In the above setting, let $\pi \in \mcal P(\mcal V)$ such that $\eta(\pi)=0$ and $\delta^0[\mcal T^\pi(W',X')] \neq 0$. Then necessarily $T^{\rho(\pi)}$ is a pseudo-cactus, and the partition $\pi$ do no identify edges or internal vertices of $\mcal T^\pi$  from different strong components of $T^{\rho(\pi)}$.
\end{proposition}

The motivations of this statement are presented next section in Proposition \ref{Th:UI}. The aim of this section is to prove the Proposition \ref{Prop:Support} and provide elements for the computation of the contribution on each strong component. 
The first lemma indicates the multiplicity of the edges labeled $w$ in the graphs that contribute in the large $N$ limit.

\begin{lemma}\label{MultiplicityInW} Assume that $T$ and $\pi \in \mcal P(\mcal V)$ are split, that $\delta^0[\mcal T^\pi(W',X')]$ is nonzero and $\eta(\pi)=0$. Then there are exactly one group of edges of multiplicity 3 labeled $w$ in $\mcal T^\pi$ within each niche corresponding to a simple edge of $T^{\tilde \rho(\pi)}$, and all other groups of edges labeled $w$ have multiplicity 2.
\end{lemma}

\begin{proof}
Recall that $\mathcal{E}^\pi_{w, 3} \subset \mathcal{E}^\pi_w$ stands for the set of edges of multiplicity equal to 3 in $\mathcal{T}_w^\pi$ and denote by $\mathcal{E}^\pi_{w, >3}$ those of multiplicity greater than 3. Since $T^\pi_w$ has no edge of multiplicity 1 (otherwise $ \delta^0 \big[ \mcal T_w^\pi(W')\big]=0$), then \eqref{Db} reformulates as 
\eq
 \frac{ l^\pi - | \mcal E^\pi_{w,3}|}2 + \big(  |\bar{\mathcal{E}}_{w,>3}^\pi| - \frac{|\mathcal{E}_{w,>3}^\pi|}{2}\big)&=&0.
\qe
The second term $\big(  |\bar{\mathcal{E}}_{w,>3}^\pi| - \frac{|\mathcal{E}_{w,>3}^\pi|}{2}\big)$  is non positive and vanishes whenever $\mathcal{E}^\pi_{w, >3} = \emptyset$. Moreover, since the number of edges in each niche is odd, and since by Lemma \ref{simpedge}  simple edges of $T^{\tilde \rho(\pi)}$ have no $x$-neighbor, each of these niches contain at least a group of edges labeled $x$ of multiplicity at least three. But if some edges forms an intra-niche pairing, so do their compagnons (recall from Definition \ref{TrondGraph} that two edges of $\mcal T$ are compagnon whenever they share the same internal vertex). Hence each niche associated to a simple edge of $T^{\tilde \rho(\pi)}$ contains at least one group of edges of multiplicity $\geq 3$ labeled $w$. Hence we get, $
l^\pi \leq |\bar{\mathcal{E}}_{w,\geq 3}^\pi|$, so that $(l^\pi - | \mcal E_{w,3}|)$ is non positive and vanishes whenever $|\mathcal{E}_{w, 3}| = l^\pi$. All together, this proves the lemma. 
\end{proof}

First we use the above lemma to relate the edges of multiplicity $1$ in $T^{\tilde \rho(\pi)}$ with the cut edges of $T^{\rho(\pi)}$.

\begin{lemma}\label{SimpleIsolated}
Assume that $T$ and $\pi \in \mcal P(\mcal V)$ are split, that $\delta^0[\mcal T^\pi(W',X')]$ is nonzero and $\eta(\pi)=0$. Let $e$ be an edge of $E$ of multiplicity $1$ in $T^{\tilde \rho(\pi)}$. Then it is a cut edge of $T^{\rho(\pi)}$ and its niche have no extra-niche identification. 
\end{lemma}

\begin{proof} Let $e$ be an edge of $E$ of multiplicity $1$ in $T^{\tilde \rho(\pi)}$ and denote in short by $\mcal N$ its niche. We first prove the second part of the lemma. It is known from Lemma \ref{simpedge} for the edges labeled $x$ in $\mcal N$ have no extra-niche identification, let us prove it for the edges labeled $w$. Lemma \ref{simpedge} also tells that all internal vertices in $\mcal N$ have always intra-identifications. Hence there is one group of internal vertices of multiplicity  at least 3 in $\mcal T^{\rho(\pi)}$, and the other groups may have multiplicity greater than or equal to 2. But Lemma \ref{MultiplicityInW} tells that the edges labeled $w$ in this niche form one group of multiplicity 3 and other groups of multiplicity at least 2. All together, this prove that there are no additional extra-niche identifications.

Let us now prove that there are no extra-niche identification of internal vertices in $\mcal N$. The following argument, referred as the argument of \emph{separation}, is use several times in the sequel. Denote by $\check V_0$ the set of internal vertices of $\mcal N$ that are identified with a vertex outside its niche and let us prove that $\check V_0=\emptyset$. Denote by $\check \pi \in \mcal P(V)$ the modification of $\pi$ obtained by isolating the vertices of $\check V_0$ in the following sense
	\eqa\label{DefSeparation}
		\check \pi = \big\{ B \setminus \check V_0, B\in \pi \big\} \sqcup \big \{ B \cap \check V_0, B\in \pi \big\}.
	\qea
This modification does not change the $\delta^0$ weight associated to the partition, that is $\delta^0[\mcal T^\pi(W',X')] = \delta^0[\mcal T^{\check \pi}(W',X')]\neq 0$. Hence by the previous section, we have $\eta(\check \pi)\leq 0$. But by definition $\eta(\check \pi )  = \eta(\pi)+ | \check V_0|$ and by hypothesis $\eta(\pi)=0$, so $| \check V_0|=0$. We have prove that there are no extra-niche identifications in the niche of $e$.

It remains to prove that $e$ is necessarily a cut edge of $T^{\rho(\pi)}$. For clarity of the presentation, assume first that $e$ is not a cut edge of $T$ (in which case it cannot be a cut edge of $T^{\rho(\pi)}$) and let us find a contradiction. We use a similar argument as in the previous paragraph, considering a modification $\check T$ of $T$ obtained by separating a vertex $v$ of $e$ as follow: we add a vertex $v'$ and decide that $e$ is adjacent to $v'$ instead of $v$. The resulting graph is connected. Moreover, separating $\check V_0:=\{v\}$ in $\pi$ as in \eqref{DefSeparation} gives a partition $\check \pi$ of $\mcal T_{\check T}$ such that $\delta^0[\mcal T_{\check T}^{\check \pi}(W',X')] =\delta^0[\mcal T^\pi(W',X')] \neq 0$ and $\eta(\check \pi )  = \eta(\pi)+1$. Hence a contradiction.

Finally, we consider the case where $e$ is a cut edge of $T$ but not of $T^{\rho(\pi)}$. Let $S$ be the disconnected graph obtained from $T$ by separating a vertex $v$ of $e$. Denote by $S_1$ the connected component of $e$ in $S$ and $S_2$ its complementary. As before, separating $v$ in $\pi$ provides a partition $\check \pi$ of the vertices of $S$ with same $\delta^0$ weight. A priori we cannot use the previous reasoning since $S$ is disconnected, but actually we get the same contradiction by considering another graph. Let $\check T$ be obtained from $S$ by identifying a vertex of $S_1$ and a vertex of $S_2$ that are identified by $\pi$ (it always exists since $e$ is not a cut edge of $T^{\rho(\pi)}$. The partition $\check \pi$ induces a partition $\check \pi'$ of the vertices of $\check T$ with same $\delta^0$ weight and such that $\eta(\check \pi')= \eta(\pi)+1$ (we have an additional vertex without changing the number of edges), obtaining the same contradiction.
\end{proof}

We now related the edges of $T^{\rho(\pi)}$ that are of multiplicity 2 in $T^{\tilde \rho(\pi)}$ with the simple cycles of $T^{\rho(\pi)}$.

\begin{lemma}\label{Cycle}
Assume that $T$ and $\pi \in \mcal P(\mcal V)$ are split, that $\delta^0[\mcal T^\pi(W',X')]$ is nonzero and $\eta(\pi)=0$. Let $e \in E$ be an edge of multiplicity $2$ in $T^{\tilde \rho(\pi)}$. Then $e$ belongs to a unique simple cycle of $T^{\rho(\pi)}$. Moreover, none of the edges or internal vertices in the niches of the edges of this cycle is identified out the union of these niches.
\end{lemma}

\begin{proof}
Let $e_0$ be an edge of $E$ of multiplicity $2$ in $T^{\tilde \rho(\pi)}$ and denote in short by $\mcal N_0$ its niche. We want to prove that $e_0$ has a unique $x$-neighbor $e_1$ and has a unique $w$-neighbor $e_{-1}$. In other words, there exists a couple of compagnon vertices $e^{(0)}_w, e^{(0)}_x$ in $\mcal N_0$ with extra-niche identifications. The latter couple is not necessarily unique, but our analysis below shows that when it is not unique then $e_0$ belong to a cycle of length two. 

The existence of a $w$-neighbor $e_{-1}$ of $e_0$ follows from Lemma \ref{MultiplicityInW}, which indicates that that the multiplicity of the edges labeled $w$ of $\mcal T^{\tilde \rho(\pi)}$  is two. Indeed, since the number of edges of each label in a niche is odd, a parity argument ensures that there exists an edge $e^{(0)}_w\in \mcal N_0$ and an edge $e^{(-1)}_w$ in the niche $\mcal N_{-1}$ of another edge $e_-1$ of $T$ that are identified.

To prove the existence a $x$-neighbor $e_1$ of $e_0$, let us prove that the compagnon $e^{(0)}_x$ of $e^{(0)}_w$ has  an extra-niche identification. Assume momentarily that $e_x^{(0)}$ has no extra-niche identification and let us find a contradiction. Since $\delta^0[\mcal T^\pi(W',X')]$ is nonzero, the multiplicity of $e_x^{(0)}$ is greater than 1, so it forms an identification which is necessarily an intra-niche identifications if the latter assumption is valid. But when two edges form an intra-niche identification, so do their compagnons. This implies that $e^{(0)}_w$  has both intra and extra-niche identification, so its multiplicity in $\mcal T^\pi$ is at least 3. This is in contradiction with Lemma \ref{MultiplicityInW} which says that this multiplicity is 2. Hence $e^{(0)}_x$ has an extra-niche identification with an edge $e^{(1)}_x$. We denote by $\mcal N_1$ the niche of $e^{(1)}_x$ and by $e_1$ the associated $x$-neighbor of $e_0$. 

We now show the unicity of the $x$-neighbor $e_1$. Let $e_1'$ be any $x$-neighbor of $e_0$, and denote by ${e'}^{(1)}_x$ and ${e'}^{(0)}_x$ two edges in the respective niches that are identified. This identification implies that the sources of $e_0$ and $e_1'$ are equal in $\mcal T^{\tilde \rho(\pi)}$. Also, the sources of the compagnons of ${e'}^{(1)}_x$ and ${e'}^{(0)}_x$ are also equal, so their targets belong to the same connected component. Hence $e_0$ and $e_1'$ form an edge of multiplicity at least two in $T^{\tilde \rho(\pi)}$. On the other hand, $\eta(\pi)=0$ implies that $T^{\tilde \rho(\pi)}$ has edges of multiplicity two. This proves that $e_1'=e_1$ since otherwise this will exhibit an edge of multiplicity greater than 2.

We now prove the uniqueness of the $w$-neighbor $e_{-1}$ of $e_0$. Recall that $e^{(-1)}_w$ and $e^{(0)}_w$ denote two edges in the respective niches forming an extra-niche identification. Assume that there is another edge ${e'}_w^{(0)}\in \mcal N_0$ forming an extra-niche identification. We then consider the compagnon ${e'}_x^{(0)}$ of this new edge ${e'}_w^{(0)}$. By the above, the $x$-neighbor is unique, so ${e'}_x^{(0)}$ has also an extra-niche identification with an edge ${e'}_x^{(1)}$ in the niche $\mcal N_1$ of $e_1$. Note that ${e'}_x^{(1)} \neq e_x^{(1)}$ otherwise $e_w^{(0)},  {e'}_w^{(0)}$ and $e_w^{(-1)}$ might form an edge of multiplicity at least 3 in $\mcal T_\omega^\pi$. Now consider the graph consisting in $e_w^{(0)}$, $  {e'}_w^{(0)}$, and the compagnons of  $  e_x^{(1)}$ and $ {e'}_x^{(1)}$. It consists in a cycle of length 4, which is a simple cycle if the vertices are pair-wise distinct. By construction, one sees directly that the edges are pairwise distinct and so are their source vertices. On the other hand, the condition $\eta(\pi)=0$ implies that the skeleton of $\mcal T_w^\pi$ is a forest. Hence the graph is not a simple cycle, which means that the targets of the edges forming this graph are identified. Hence $e_w^{(0)}$ and $e_w^{(1)}$ are identified and so $e_{-1} = e_1$. This prove the uniqueness of $e_{-1}$, and that $e_{-1} = e_1$ when $e_0$ and $e_{-1}$ have more than one group of edges identified to form extra-niche identification (this fact will be relevant later on). 

We can now prove the lemma. Starting with $e_0$ we construct a sequence $e_i,i=0\etc r$ of edges in $T$ such that, if $i$ is even, $e_i$ and $e_{i+1}$ are $x$-neighbors, and if $i$ is odd they are $w$-neighbors, until we come back to an edge we have already visited ($r = \min\{ s  | e_s \in \{ e_0\etc e_{s-1}\}$). By uniqueness of extra-niche neighbors, we necessarily have $e_{r} = e_0$ (and so $r$ is even).

Let us denote by $\mcal C = \{e_0 \etc e_{r-1}\}$ the set of edges forming this cycle and by $\mcal N(\mcal C)$ the union of the niches of elements of $\mcal C$. By construction, the edges of $\mcal N(\mcal C)$ are not identified outside of $\mcal N(\mcal C)$. We then deduce that tis property holds for the internal vertices of $\mcal N (\mcal C)$ with the usual argument of separation. We modify $\pi$ into $\check \pi$ to separate the set $\check V_0$ of vertices of $\mcal N(\mcal C)$ that are identified outside of $\mcal N(\mcal C)$. This does not change the $\delta^0$ weight since it does not modify the multiplicity of the edges. Hence it produces a partition such that $\eta(\check \pi)\geq 0$ and $\eta(\check \pi) = \eta(\pi)+|\check V_0|$, the assumption $\eta(\pi)$ implying that $\check V_0 = \emptyset$. This proves the second part of the lemma. 

Finally we can prove that $e_0$ belongs to a unique cycle of $T^{\rho(\pi)}$. Let $S$ be the graph obtained by removing to $T^{\rho(\pi)}$ the edges of $\mcal C$, and with a small abuse use $\mcal C$ to refer to the subgraph of $T^{\rho(\pi)}$ formed by these edges. Note that $e_0$ belongs to a unique cycle if and only if the connected components of $S$ have exactly one vertex in $\mcal C$. On the other hand, the presence of a connected component $C$ with at least two vertices in $\mcal C$ would yield a contradiction thanks to the separation argument: one modifies $T^{\rho(\pi)}$ by separating a vertex common to $\mcal C$ and $C$, which does not disconnect the graph, producing a quotient of higher but bounded contribution. Hence $e_0$ belong to a unique cycle, which concludes the proof of the lemma.
\end{proof}

Since Condition \eqref{Db} says that the edges of $T^{\tilde \rho(\pi)}$ are of cardinality 1 or 2, the two above lemmas show Proposition \ref{Prop:Support}.

\subsection{Asymptotic expression of the traffic distribution}

 Given a partition $\rho_0 \in \mathcal{P}({V})$ such that $T^{\rho_0}$ is a pseudo-cactus and a partition of the auxiliary graph $\pi \in \mcal P(\mcal V)$, we write $\rho(\pi)=\rho_0$ is a shortcut to say that $\pi$ is a split partition such that $\delta^0[\mcal T^\pi(W',X')]\neq 0$, $\eta(\pi)=0$ and the restriction of $\pi$ on $V$ is $\rho(\pi) = \rho_0$.

 The previous section proves that for any reference graph $T$ we have the asymptotic expression
	\eqa\label{FormulaSec4}
	 \lefteqn{\tau_N \big[ T( \mbf Y) \big]  } \\
		 &=&   \sum_{\substack{\rho_0 \in \mathcal{P}({V}) \mrm{\ s.t.}\\ T^{\rho_0} \mrm{\ pseudo \ cactus}}}    \ \sum_{\substack{\pi \in \mathcal{P}({V}) \\ \rho(\pi) = \rho_0}}
		   \Psi^\pi \delta^0\big[\mcal T^\pi(\Gamma_w, \Gamma_x)\big]\delta^0[\mcal T^\pi(W',X')] + o(1),\nonumber
	\qea
where we recall \eqref{DefPsi} and Definition \ref{Def:delta}: 
\eq
	\Psi^\pi =  \psi_0^{|V_0^\pi| -\sum\limits_{e\in E} \frac{\mbf n(e)-1}{2}}  \psi_2^{|V_2^\pi|}  \psi_1^{|V_1^\pi|}, \quad \delta^0\big[\mathfrak T(\mbf A) \big]=  \esp\left[ \prod_{e \in\mathfrak  E} \tilde A_{\gamma(e)}\big( \Phi(e)  \big) \right],
\qe
$\psi_i = \frac{N_i}{ {N}}$ for $i=0,1,2$ and $\Phi: \mathfrak V \to [N]$ is split, injective and uniformly distributed independently of $\mbf A$.

It may be useful to compare Formula \ref{FormulaSec4} with the following consequence of \cite[Chapter 6]{Male2020} and \cite[Part II]{CDM24}. In the proposition below, we call \emph{well-oriented} (w.o.) pseudo-cactus a pseudo-cactus for which the edges of each simple cycle follow a same orientation along their cycle.

\begin{proposition}\label{Th:UI} Let $\mbf A=(A_j)_{j \in J}$ be a family of $N$ square random matrices that are unitarily invariant in law.  Assume that $\mbf A$ converges in non commutative distribution and satisfies the asymptotic fractorization property. Denote by $\mbb J_N$ the matrix whose all entries are $\frac 1 N$. Then for all $\beta\in \mbf C$ the collection $\mbf A+\beta \mbb J_N := (A_j + \beta \mbb J_N)_{j\in J}$ converges in traffic distribution. Moreover, for any test graph $T$ in the variables $\mbf a=(a_j)_{j\in J}$ and any partition $\rho_0$ of the vertex set of $T$, we have
	\eq
	 \tau_N \big[ T(  \mbf A+\beta \mbb J_N) \big]  	 &\limN&   \sum_{\substack{\rho_0 \in \mathcal{P}({V}) \mrm{\ s.t.}\\ T^{\rho_0} \mrm{\ w.o. \ pseudo \ cactus}}}  	\prod_{ C \in \mcal {SC}(T^\rho_0)} \tau^0(C),
	\qe
where for any $C \in  \mcal {SC}(T^\rho)$, if $C$ is a cut-edge then $\tau^0(C) = \beta$, and if $C$ is a simple cycle then $\tau^0(C) = \kappa_n(a_1\etc a_n)$ is the $n$-th free cumulant function applied to the limit of the matrices $A_1\etc A_n$ along the cycle of $C$. \end{proposition}

The rectangular analogue can be deduced from the computations of \cite{Zit24} (explaining our scaling factor) and the real analogue (with orthogonal invariance) from \cite{Au16}. The method consists hence in  re-writing \eqref{FormulaSec4} by summing over the partitions of the reference graph $T$ rather than $\mcal T = \mcal T_T$ and  exhibiting some factorization structure with respect to the strong components.

The latter is the motivation for the second part of the statement in Proposition \ref{Prop:Support}: it shows that these partitions $\pi\in \mcal P(\mcal V)$ that contribute can be factorized with respect to the niches of the strong components $S$ of $T^{\rho_0}$ in the following way. For any strong component $S \in {SC}(T^\rho_0)$, and any $\pi$ such that $\rho(\pi)=\rho_0$, denote by $\pi_S$ the restriction of $\pi$ to the vertices in $S$ and its niche vertex. Then since $\pi$ does not identify extra-strong components internal vertices, $\pi$ is the finest of all partitions in  $\mcal P(\mcal V)$ such that the internal blocks of $\pi_S$ are contained in blocks of $\pi$, for any $S \in {SC}(T^{\rho_0})$. 

In the three following subsections, we consider a strong component $S=S(\rho_0)$ of $\mcal T^{\rho_0}$ of a given type, i.e. a cut-edge, a length 2 cycle or a higher length cycle. In each case, we describe the partition $\pi_S$ and the subgraph of $\mcal T^\pi$ induced by the niches of the edges of $S$, and gives an illustration in Figure \ref{F02_Pad}. We denote from now by $n_1(S) \etc n_{L_S}(S)$ an enumeration of the edge labels in $S$, with the shortcuts $n(S):=n_1(S)$ if $L_S\leq 2$ and $n'(S):=n_2(S)$ if $L_S=2$. We write ''$n=2k+1$`` with all variations of indices.

\subsubsection{Focus on cut edges}\label{Focus1} We denote by $\binom{p}{q} = \frac{p!}{q!(p-q)!}$ the usual binomial coefficient counting the number of choice of $q$ element among $p$, and by $\mcal P_2(n)$ the set of pairings of $n$ elements.  Recall that if $\xi$ denotes a standard real Gaussian random variable, then $\esp[\xi^n] = |\mcal P_2(n)|$. For each $n\geq1$, we recall that $h_n:x\mapsto x^n$ denotes the $n$-th power function. Then the Gaussian integration  formula reads $ n \esp[h_{n-1}(\xi)] = \esp[h'_n(\xi)]$. 

Assume that $S\in {SC}(T^\pi)$ consists in a cut edge $e$. Lemma \eqref{MultiplicityInW} says that $\pi_S$ identifies 3 internal vertices to form a first block and pairs the other internal vertices. Hence $n(S)\geq 3$ otherwise we cannot form a group of 3 internal edges of same label. 
We have a total number of  
	$$\binom{n(S)}3\big|\mcal P_2\big(n(S)-3\big)\big| = \binom{n(S)}3 \esp[\xi^{n(S)-3}] =\esp[h'''_{n(S)}(\xi)]/6 $$
 partitions as above, any of them having $k(S)$ internal vertex blocks.

\subsubsection{Focus on length 2 simple cycles}\label{Focus2} Assume that $S$ consists in a cycle of length 2. Lemma \ref{MultiplicityInW} implies that the internal vertices are paired. It must have at least one block formed by an internal vertex of the niche of each edge, but since there is an odd number of vertices in each niche, each pairing satisfies this property. We have a total number of 
 	$$ \big|\mcal P_2\big(n(S)+n'(S)\big)\big| =  \esp[\xi^{n(S)+n'(S)}]=  \esp \big[  h_{n(S)(\xi)} h_{n'(S)}(\xi) \big]$$
partitions $\pi_S$ as above, any of them having with a total of $k(S) +k'(S) +1$ 
internal vertex blocks.

\subsubsection{Focus on higher lengths}\label{Focus3} Assume that $S$ consists in a simple cycle of length $L\geq 3$ of extra-niche successive neighbors. Constructing the cycles in the proof of Lemma \ref{Cycle}, we have shown that in the niche of each edge of $S$ there is an edge labeled $x$ (the one realizing the neighboring) such that $\pi$ identifies the targets of all these edges. This forms a first  block of $\pi_S$, that we refer as the central block (note that it contains at least one internal vertex from the niche of each edge of $S$). This central block actually cannot contain more that one vertex from each niche: otherwise one sees easily that this will produce an edge labeled $w$ with multiplicity greater than 2 with the usual compagnon argument. Moreover, the proof of Lemma \ref{Cycle} shows that each niche has a single  edge labeled $w$ forming an extra-niche identification unless the cycle is of length 2. This implies the same property for edges labeled $x$ (since the compagnon of edge labeled $x$ forming extra-niche identification also form an extra-niche identification by the multiplicity 2 constraint).  The conclusion is that $\pi$ consists in the central block together with pairings of the remaining vertices in order to pair the edges labeled $w$. 
 
To chose $\pi_S$ we can first chose its central block by choosing one internal vertex in each niche, and then we chose intra-niche pairing of the remaining vertices, which gives a total number of 
 	$$  \prod_{\ell=1}^{L_S} n_\ell(S)  \big|\mcal P_2\big(n_\ell(S)-1\big)\big| = \prod_{\ell=1}^{L_S} n_\ell(S)   \esp[\xi^{n_\ell(S) - 1}] = \prod_{\ell=1}^{L_S} \esp\big[ h'_{n_\ell(S)}(\xi) \big] $$
possibilities for partition $\pi_S$ satisfying the above condition, any of these partitions having with a total of 
	$1+\sum_{\ell=1}^{L_S}k_\ell(S)$
internal vertex blocks.

  \begin{figure}[h]
    \begin{center}
     \includegraphics[scale=.75]{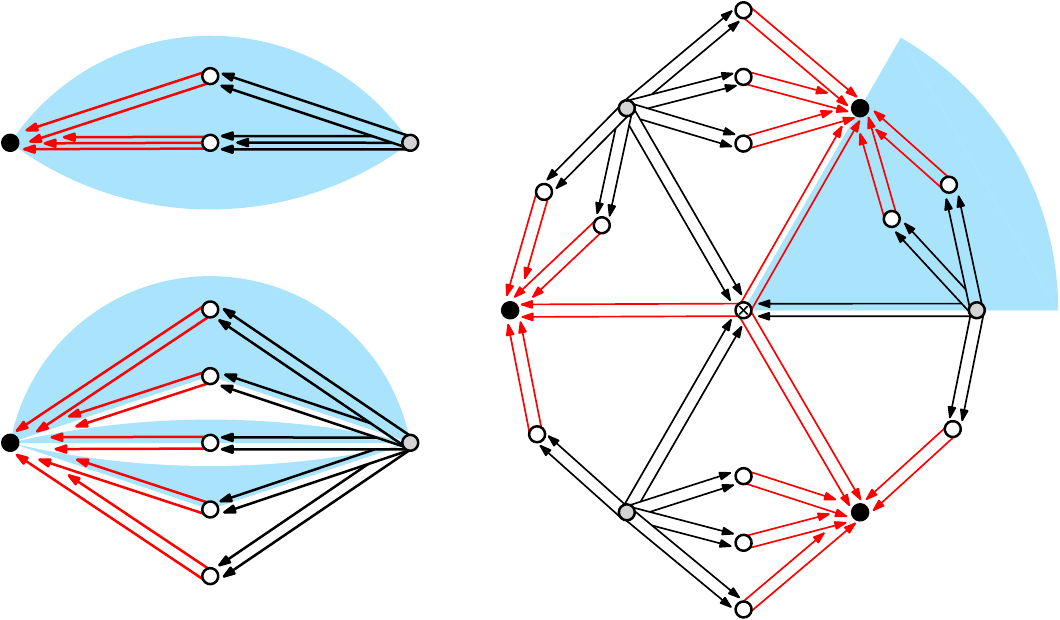}
    \end{center}
    \caption{Three types of strong components that contribute in the limiting traffics distribution: the upper leftmost figure represents the strong component associated to a cut edge of $\mcal T^{\rho0}$, the bottom leftmost stands for a length two cycle, and the rightmost one for a cycle of length 6. The light blue region represents the niche of an edge labeled by $x\mapsto x^5$.}
    \label{F02_Pad}
  \end{figure}

\subsubsection{Conclusion} 

Recall that two test graphs are \emph{isomorphic} whenever there exists a bijection between their vertex set that preserves adjacency and labels. 

For a partition $\rho_0$ of $V$ such that $T^{\rho_0}$ is a pseudo-cactus and the previous section shows that for any  $\pi$ such that $\rho(\pi) = \rho_0$, the isomorphic class of $\mcal T^{\pi}$ depends only on $\rho_0$, not on $\pi$. On the other hand, the summand in the sum over $\pi$ in  \eqref{FormulaSec4} is a function of the isomorphism class of $\mcal T^\pi$, so it is a function of $\rho_0$. In this section we write explicitly the dependance in $\rho_0$ except for the contribution of the profiles $ \delta^0\big[\mcal T^\pi(\Gamma_w, \Gamma_x)\big]$ which is considered later. 

To write $\delta^0[\mcal T^\pi(W',X')] $ in terms of $\rho_0$ recall that for each cut edge of $T^{\rho(\pi)}$, there is in $\mcal T^\pi$ a group of edges of multiplicity 3 in each variables and the other groups are all of multiplicity 2.  Denote by 
\begin{itemize}
 	\item $\mcal C_1(\rho_0)$ the set of cut edges of $T^{\rho_0}$, $ c_1(\rho_0)=|\mcal C_1(\rho_0)|$,
	\item $m_w^{(3)} = \esp[W'(1,1)^3]$ and $m_x^{(3)} = \esp[X'(1,1)^3]$ the third moments of the matrix entries.
\end{itemize}
Recall that the variables are normalized ($\esp[W'(1,1)^2] = \esp[X'(1,1)^2]=1$).
 We hence deduce that $\delta^0[\mcal T^\pi(W',X')]= 0$ if $T^{\rho_0}$ has a cut edge $S$ with label $n(S)=1$, and otherwise
	\eq
		\delta^0[\mcal T^\pi(W',X')] & := & \mathbb{E}\left[\prod_{e\in \mathcal{E}^\pi_w}\tilde W'(\varphi(e)) \times \prod_{e\in \mathcal{E}^\pi_x}\tilde X'(\varphi(e)) \right] =   \big( m_w^{(3)} m_x^{(3)}\big)^{c_1(\rho_0)}
	\qe
which is independent of the edge labels $Y_n$ of $T$ with index $n\geq 3$.

In order to write $\Psi^\pi$ as a function of $\rho_0$ note first that $|V_1^\pi| = |V_1^{\rho_0}|$ and $|V_2^\pi| = |V_2^{\rho_0}|$ by definition. Moreover, the number of internal vertices does not depend on $\pi$: indeed, denote
 \begin{itemize}
	\item $ C_2(\rho_0)$  the set of simple cycles of length 2 of $T^{\rho_0}$, $c_2(\rho_0) = |C_2(\rho_0)| $,
	\item $C_3(\rho_0)$ the set of higher length simple cycles, and $c_3(\rho_0) = |C_3(\rho_0)| $.
\end{itemize}
Sections \ref{Focus1}, \ref{Focus2} and \ref{Focus3} yield
\eq
	|V_0^\pi| & = &  \sum_{  S\in C_1(\rho_0) } k(S)  +   \sum_{  S\in C_2(\rho_0) } \big(k(S) +k'(S)  + 1\big)\\
	& &  +  \sum_{  S\in C_3(\rho_0) }  \left(1+ \sum\limits_{\ell=1}^{L_S} k_\ell(S) \right)=  c_2+c_3+\sum_{e\in E} \frac{\mbf n(e)-1}2 .
\qe
So the definition of $\Psi^\pi$ gives
\eq
\Psi^\pi   := \psi_0^{|V_0^\pi| -\sum\limits_{e\in E} \frac{\mbf n(e)-1}{2}}  \psi_2^{|V_2^\pi|}  \psi_1^{|V_1^\pi|} = \psi_0^{c_2 + c_3} \psi_1^{|V_1^{\rho_0}|}  \psi_2^{|V_2^{\rho_0}|}, 
\qe 
where $V_i^{\rho_0}$ is the set of vertices of $T^{\rho_0}$ in $V_i$. As announced, this expression is a function of $ \rho_0$. It also does not depend on the labeling.

The number of partitions $\pi$ such that $\rho(\pi)=\rho_0$ is given by the count of Subsections \ref{Focus1} to \ref{Focus3}. We recall $h_n:x\mapsto x^n$ and $\xi$ a standard real Gaussian random variable. With the above computations and \eqref{FormulaSec4}, we have finally obtained the following asymptotic formula.

\begin{lemma}\label{Lem:PreHeuristic} Let $\mbf Y$ be the family of profiled Pennington-Worah matrices defined in \eqref{Yk}. Then under our assumption,  for any reference test graph $T_{\mbf n} = (V,E,\mbf n)$, with  $\mcal T_{\mbf n} = \mcal T_{T_\mbf n}$ the associated auxiliary test graph and $\pi_0\in \mcal P(\mcal V)$ such that $\rho(\pi_0) = \rho_0$, we have
	\eqa
	 \tau_N \big[ T_{\mbf n}( \mbf Y) \big] 
		& = &  \sum_{\substack{\rho_0 \in \mathcal{P}({V}) \\ T_{\mbf n}^{\rho_0} \mrm{\ pseudo \ cactus}}} 
		 \psi_2^{|V_2^{\rho_0}|}  
		 \psi_1^{|V_1^{\rho_0}|} \delta^0\big[\mcal T_{\mbf n}^{\pi_0}(\Gamma_w, \Gamma_x)\big] 				\nonumber\\
		 & & \times  \left (  \frac{ m_w^{(3)} m_x^{(3)}}6\right)^{c_1(\rho_0)} \prod_{S\in \mcal C_1(\rho_0)}  \esp\big[   h'''_{n(S)}(\xi) \big] \nonumber\\
		 & & \times \   \psi_0 ^{c_2(\rho_0)} \prod_{S\in \mcal C_2(\rho_0)}    \esp \big[  h_{n(S)(\xi)} h_{n'(S)}(\xi) \big]    \nonumber    \\
		 & & \times \ \psi_0^{c_3(\rho_0)}\prod_{ S \in \mcal C_3(\rho_0)}   \prod_{i=1}^{L_S}  \esp\big[ h'_{n_i(S)}(\xi) \big]   +o(1).\label{Eq:PreHeuristic}
 	\qea
\end{lemma}

\section{Construction of the asymptotic equivalent}

In this section, we analysis the expression \eqref{Eq:PreHeuristic}, using Proposition \ref{Th:UI} and  technics from traffic probability \cite{Male2020} in order to  construct three explicit families of matrices $\mbf Y^{\mrm{lin}}$, $\mbf Y^{\mrm{per}}$ and $\mbf B$ indexed by $\mbb C[y]$ such that, for the collections restricted to odd polynomials, $\mbf Y$ has the same limiting traffic distribution as 
	$$\mbf Y^{\mrm{lin}}+ \mbf Y^{\mrm{per}}  + \mbf B := \big( Y^{\mrm{lin}}(h) + Y^{\mrm{per}}(h)+B(h) \big)_{h \in \mbb C[y]}.$$
Moreover, each matrix is a linear function of its argument $h\in \mbb C[y]$. The construction of each collection comes from the analysis the contributions of each type of strong component in  \eqref{Eq:PreHeuristic}.

The limiting traffic distribution of  $\mbf B$ has the same expression as \eqref{Eq:PreHeuristic} if we set to zero the contributions that are not associated to the cut-edge set $\mcal C_1(\rho_0)$, see Lemma \ref{Lem:DistB}. The matrices of $\mbf B$ are deterministic and their entries are of order $N^{-1}$ (they are called of Boolean type in \cite{Male2020}).

The construction of the collection $\mbf Y^{\mrm{lin}}$ is motivated by the contribution from $\mcal C_3(\rho_0)$ in \eqref{Eq:PreHeuristic}. The matrices of $\mbf Y^{\mrm{lin}}$ are obtained by applying profiles to the matrix $Y_1 = \frac{WX}{\sqrt{N}\sqrt{N}}$. Comparing with Proposition \ref{Th:UI}, it may be useful to recall that the product of two independent Ginibre matrices converges toward a non commutative random variables whose free cumulants are constant (they do not depend on the order of the cumulant), called a free Poisson variable. By Shlyakhtenko \cite{SHL}, since we can expect that a good notion of \emph{free Poisson variable over the diagonal} holds to describe the asymptotic of $\frac{WX}{\sqrt{N}\sqrt{N}}$ in canonical terms.

Therefore the linear matrix $\mbf Y^{\mrm{lin}}$ comes also with contributions for length 2 cycle that we must subtract from the length 2 cycle contribution \eqref{Eq:PreHeuristic} in order to surmise the perturbation family $ \mbf Y^{\mrm{per}}$. Recall that for a collection of non commutative random variables, all free cumulants of order greater than 2 vanish if and only if the collection is circular or semi-circular, which are the limit GOE and Ginibre matrix ensembles. Shlyakhtenko proves in \cite{SHL} the analogue for variance profiled matrices. The consequence is that to construct $ \mbf Y^{\mrm{per}}$ it suffices to understand a covariance structure. This is in particular the moment where we switching from reference test graphs to test graphs labeled by the Hermite polynomials.

Each case relies a same lemma stated in the following subsection.

\subsection{A property of the function $\delta^0$}

In Section \ref{PWModel}, after the definition of the auxiliary graph $\mcal T_T$, we use the \emph{substitution property} while replacing the edge of a test graph by graph operations. While this property is obvious for the evaluation of the combinatorial traces, it is not longer true for the injective trace. The lemma below shows that the substitution property can be applied for the map $\delta^0$ evaluated in bounded matrices. We restrict our statement to the situation we meet later on.

\begin{lemma}\label{ProfileTrick} $\mbf A=(A_\omega)_{\omega \in \Omega}$ a family of $\mbf s$-rectangular random matrices with bounded entries and let $\mathfrak T = (\mathfrak V, \mathfrak E, \gamma)$ be a test graph. Assume that $\mathfrak T$ has an edge $e_0 \in V_2 \times V_1$ associated to a matrix of the form $A_{\gamma(e_0)} = g_{\sigma}(\mbf A)$, where $g = (\mathfrak S, \mrm{in}, \mrm{out} )$. Assume that $\mrm{in}$ is the only vertex of $\mathfrak S$ in $V_2$, and  $\mrm{out}$ is the only vertex of $\mathfrak S$ in $V_1$. Let us denote $\mcal T$ the test graph obtained from $\mathfrak T$ by replacing the edge $e_0$ by the graph $\mathfrak S$, identifying the input or $g$ with the source of $e_0$, as well as the output of $g$ with the source of $e_0$. Then, denoting $R = N^{-v_0}g_{\mathfrak S}(\mbf A)$ where $v_0$ is the number of internal vertices of $\mathfrak S$,  we have as $N$ goes to infinity
$$ \delta^0\big[ \mathfrak T(\mbf A) \big] =  \delta^0\big[ \mathfrak T'(\mbf A, R) \big]+o(1).$$
\end{lemma}

\begin{proof} Recall that for $\Phi$ a uniform split injection $\mathfrak V\to [N]$  distributed independently of $\mbf A$, denoting $\Phi(e) := \big(\Phi(w), \Phi(v) \big)$ when $e=(v,w)$, we have
	\eq
		 \delta^0\big[\mathfrak T (\mbf A)\big]  :=    \esp\Big[ \prod_{e\in \mathfrak E} A_{\gamma(e)}\big( \Phi(e) \big) \Big]. 
	\qe
Setting $\mathfrak V =\mathfrak V_0\sqcup \mathfrak V_1\sqcup \mathfrak V_2$, the boundedness of the matrix entries implies
\eqa
	\lefteqn{ \delta^0\big[\mathfrak T (\mbf A)\big]  } \nonumber\\
	 & = & \frac 1 { N_1^{|\mathfrak V_1|} N_2^{ |\mathfrak V_2| } } \sum_{ \substack{ \phi_1:\mathfrak V_1\to [N_1] \\ \mrm{injective}}}  \sum_{ \substack{ \phi_2:\mathfrak V_2\to [N_2] \\ \mrm{injective}}}\label{ExplainDeltaZero} \\
	 & & \Bigg(  \frac 1 {N_0^{|\mathfrak V_0| } }\sum_{\phi_0 : \mathfrak V_0 \to [N_0]} \prod_{e\in \mathfrak  E} \upiota_{(1, 2)}\big(A_{\gamma(e)}\big)\big( \phi_{0,1,2}(e) \big)\Bigg)  +o(1)\nonumber
\qea
where is the last formula $\phi_{0,1,2}$ coincide with $\phi_i$ on $V^{\rho_0}_i$ for each $i=0,1,2$. 
Recall that $ \upiota_{(1, 2)}$ denotes the canonical injection of $\mrm M_{N_1,N_2}(\mbb R) \to \mrm M_{N,N}(\mbb R)$. By assumption, denoting $\mathfrak S= (V_{\mathfrak S}, E_{\mathfrak S}, \gamma_{\mathfrak S})$, we have
$$\upiota_{(1, 2)}\big(A_{\gamma(e)}\big)\big( \phi_{0,1,2}(e_0) \big) = \frac 1 {N^{v_0}} \sum_{\phi'_0 : \mathfrak V_{\mathfrak S} \to [N_0]} \prod_{e\in \mathfrak  E_{\mathfrak S}} \upiota_{(1, 2)}\big(A_{\gamma_{\mathfrak S}(e)}\big)\big( \phi'_{0,1,2}(e) \big),$$
where $ \phi'_{0,1,2}$ is defined as $ \phi_{0,1,2}$ with $\phi'$ instead of $\phi$.  We therefore can write our expression in terms of $\mcal T$. Denoting $\mcal V = \mathfrak V_1 \sqcup \mathfrak V_2 \sqcup \mcal V_0$ its vertex set, $\mcal E$ its edge set, and keeping the notation $\gamma$ for the label map,
\eq
	\lefteqn{\frac 1 {N_0^{|\mathfrak V_0| } }\sum_{\phi_0 : \mathfrak V_0 \to [N_0]} \prod_{e\in \mathfrak  E} \upiota_{(1, 2)}\big(A_{\gamma(e)}\big)\big( \phi_{0,1,2}(e) \big)}\\
	& = & \frac 1 {N_0^{| \mcal V_0| } }\sum_{\phi_0 : \mcal  V_0 \to [N_0]} \prod_{e\in \mathfrak  E} \upiota_{(1, 2)}\big(A_{\gamma(e)} \big)\big( \phi_{0,1,2}(e) \big)
\qe
By the same reasoning as for showing \eqref{ExplainDeltaZero} in the reverse sense, we get that the latter expression equals $  \delta^0\big[ \mcal T(\mbf A, R) \big]$ up to a negligible error resulting from the  replacement of the  map on $\mathcal V_0$ by an injective map.
\end{proof}

Our strategy is to apply the above lemma in each edge of $T^{\rho_0}$ in  Lemma \ref{Lem:PreHeuristic} reducing the complicated structure of profiled PW matrices to simpler matrix ensembles. 

\subsection{The Boolean type deterministic deformation}

In the context of Lemma \ref{Lem:PreHeuristic}, let $e\in T^{\rho_0}$ be a cut-edge with label $\mbf n(e)\geq 3$. Section \ref{Focus1} describes the subgraph $\mathfrak S = \mathfrak S(e)$ of $\mcal T^{\pi_0}$ induced by the niche of $e$. Denote $g_{\mathfrak S}$ the graph monomial whose test graph is $\mathfrak S$ and whose input and ouput are the vertices of this graph in $V_2$ and $V_1$ respectively. The graph $\mathfrak S$ has $v_0=1 + \frac {\mbf n(e)-3}2$ internal vertices, corresponding to the $\frac {\mbf n(e)-3}2$ internal vertex pairing  plus one group of 3 vertices. The corresponding identification for the edges implies that the following expression holds
		\eqa
		\lefteqn{N^{-v_0} g_{\mathfrak S}(\Gamma_w, \Gamma_x) }\nonumber\\
		& = &  \bigg[ \Big( \frac 1 N  \sum_{\ell=1}^{N_0} \Gamma_w(i,\ell)^3 \Gamma_x(\ell,j)^3 \Big) \times \Big(  \frac 1 N  \sum_{\ell=1}^{N_0} \Gamma_w(i,\ell)^2 \Gamma_x(\ell,j)^2 \Big)^{\frac{\mbf n(e)-3}2} \bigg]_{i,j}\nonumber\\
		& =& \Lambda_3 \circ \big(  \sqrt[\circ]{ \Lambda_2}^{ \circ  (\mbf n(e)-3) }\big).\label{Eq:EntryCutEdge}
		\qea
where we have set
\eqa\label{Eq:Lambda}
		\Lambda_\ell:= N^{-1} \Gamma_w^{\circ \ell} \times \Gamma_x^{\circ \ell }, \quad \ell=1,2,
\qea
and used the concise notations $A^{\circ n} = (A_{p,q}^n)_{p,q}$ and $ \sqrt[\circ]{ A} = (\sqrt{A_{p,q}})_{p,q}$.  Note that with $h_n:x\mapsto x^n$, for any matrix $A$ we have $h_n'''[\{ xA \}]=h'''_n(x) A^{\circ (n-3)} $. 
 We hence propose to introduce the following collection of matrices.

\begin{lemma}\label{Lem:DistB} Let $\mbf B = \big( B(h) \big)_{h\in \mbb C[y]}$ be the collection of $N_1\times N_2$ deterministic matrices defined as follows: for $\xi$ a standard real Gaussian random variable, 
	\eq
		B(h) & := &   
		 \frac {m_w^{(3)} m_x^{(3)}}{6N}  \Lambda_3\circ  \esp\bigg[  h^{'''}\Big[ \Big\{ \xi \sqrt[\circ]{ \Lambda_2}  \Big\} \Big] \bigg]
	\qe
Then $\mbf B $ converges in traffic distribution and  for any reference graph $T_{\mbf n}=(V,E,\mbf n)$, with same notations as \eqref{Eq:PreHeuristic},   we have 
	 \eq
	 	\tau_N \big[ T_{\mbf n}( \mbf B) \big]  \nonumber
		& = &  \sum_{\substack{\rho_0 \in \mathcal{P}({V}) \\ T_{\mbf n}^{\rho_0} \mrm{\ tree}}} 
		  \psi_2^{|V_2^{\rho_0}|}  
		 \psi_1^{|V_1^{\rho_0}|} \delta^0\big[\mcal T_{\mbf n}^{\pi_0}(\Gamma_w, \Gamma_x)\big]\nonumber\\
		 & & \times  \left (  \frac{ m_w^{(3)} m_x^{(3)}}6\right)^{|E|} \prod_{e\in E}  \esp\big[   h'''_{n(e)}(\xi) \big].\nonumber
\qe
\end{lemma}

Note that since all the edges of a tree are cut edges, then $C_1(\rho_0) = E$ and the above expression coincides with \eqref{Eq:PreHeuristic} when $T^{\rho_0}$ is a tree. 

\begin{proof} For any edge $e\in E$ and any split injective map $\phi:V\to[N]$, 
	\eq
	\lefteqn{ \upiota_{(1, 2)}\big( B( h_{\mbf n(e)} ) \big)\big( \phi(e) \big) }\\
	& = & \frac{m_w^{(3)} m_x^{(3)}}{ 6 N} \esp\big[  h_{\mbf n(e)}'''(\xi) \big] \upiota_{(1, 2)}\Big(  \Lambda_3 \circ  \sqrt[\circ]{ \Lambda_2}^{ \circ  (\mbf n(e)-3) } \Big)\big( \phi(e) \big).
	\qe
Therefore, writing $\tau_N$ in terms of the injective distribution and using the definitions, we have 
 \eq
	 	\tau_N \big[ T_{\mbf n}( \mbf B) \big] & = & \frac 1 N \sum_{\rho_0 \in \mcal P(V)}  \sum_{ \substack{ \phi:V\to[N]  \\ \mrm{injective} }}  \prod_{e\in E} \upiota_{(1, 2)}\big(  B( h_{\mbf n(e)} ) \big)\big( \phi(e) \big) \\
		& = &   \frac 1 N \left( \frac{m_w^{(3)} m_x^{(3)}}{ 6N}\right)^{|E|}  \prod_{e\in E}  \esp\big[  h_{\mbf n(e)}'''(\xi) \big] \\
		& & \times \sum_{\rho_0 \in \mcal P(V)} \sum_{ \substack{ \phi:V\to[N]  \\ \mrm{injective} }}  \prod_{e\in E} \upiota_{(1, 2)}\Big(  \Lambda_3 \circ  \sqrt[\circ]{ \Lambda_2}^{ \circ  (\mbf n(e)-3) } \Big)\big( \phi(e) \big)
\qe
We can hence substitute the subgraphs  $\mathfrak S(e)$ defined in the beginning of this section and  apply Lemma \ref{ProfileTrick} for each vertex $e$ of $T^{\rho_0}$, getting
\eq
		\lefteqn{ \sum_{ \substack{ \phi:V\to[N]  \\ \mrm{injective} }}  \prod_{e\in E} \upiota_{(1, 2)}\Big(  \Lambda_3 \circ  \sqrt[\circ]{ \Lambda_2}^{ \circ  (\mbf n(e)-3) } \Big)\big( \phi(e) \big) }\\
		& = &   \big(N_{1} \big)_{|V_1^{\rho_0}|}  \big(N_{2} \big)_{|V_2^{\rho_0}|}  \esp\bigg[   \prod_{e\in E} \upiota_{(1, 2)}\Big( N^{-v_0} g_{\mathfrak S(e)}(\Gamma_w, \Gamma_x) \Big)\big( \Phi(e) \big) \bigg]\\
		&  = &   N^{|V_1^{\rho_0}| + |V_2^{\rho_0}|} \psi_1^{|V_1^{\rho_0}|}  \psi_2^{|V_2^{\rho_0}|} \big(1+o(1)\big)   \delta^0\big[\mcal T_{\mbf n}^{\pi_0}(\Gamma_w, \Gamma_x)\big],
\qe
We hence have
 \eq
	 	\tau_N \big[ T_{\mbf n}( \mbf B) \big] & = &\left( \frac{m_w^{(3)} m_x^{(3)}}{ 6 }\right)^{|E|}  \prod_{e\in E}  \esp\big[  h_{\mbf n(e)}'''(\xi) \big] \\
		& &\times \sum_{\rho_0 \in \mcal P(V)}  N^{\eta_B(\rho_0)} \psi_1^{|V_1^{\rho_0}|}  \psi_2^{|V_2^{\rho_0}|}  \delta^0\big[\mcal T_{\mbf n}^{\pi_0}(\Gamma_w, \Gamma_x)\big]\big( 1 + o(1) \big).
\qe
where $\eta_B(\rho_0)=-1-|E|+|V^{\rho_0}|$. By Lemma \ref{eq_graph}, $\eta_B(\pi)\leq0$ with equality if $T^{\rho_0}$ is a tree. We hence get the expected asymptotic expression.

\end{proof}

\subsection{The linear collection $\mbf Y^{\mrm{lin}}$}\label{Sec:Ylin}

We follow the same strategy as in the previous section, with the difference that we want to compare our expression with the expression of the linear model. In the context of Lemma \ref{Lem:PreHeuristic}, let $e$ be an edge of $ T^{\rho_0}$ contained in a simple cycle of length greater than 2. Section \ref{Focus1} describes the subgraph of $\mcal T^{\pi_0}$ induced by the cycle, so in particular the subgraph induced by the niche of $e$. If $\mbf n(e)=1$, the niche of $e$ consists in two edges forming the extra-niche identification. Assume that $\mbf n(e)$ is greater than one. We denote by $\mathfrak S(e)$ the subgraph generated by the other edges, that form intra-niche identifications.
 
The graph $\mathfrak S(e)$ has one vertex $\mrm{out}$ in $V_1$, on vertex $\mrm{in}$ in $V_2$,  $v_0=\frac {\mbf n(e)-1}2$ internal vertices, an edge of multiplicity two labeled $w$ between each internal vertex and $\mrm{out}$, and an edge of multiplicity two labeled $x$ between the input and $\mrm{in}$. Comparing with the context of Lemma \ref{ProfileTrick}, note that the graph monomial $g_{\mbf n(e)} = \big( \mathfrak S(e), \mrm{in} , \mrm{out}\big) $ satisfies 
	\eqa\label{Eq:GraphProofLinear}
	N^{-v_0}g_{\mbf n(e)}(\Gamma_w, \Gamma_x) = \Big( \sqrt[\circ]{ \Lambda_2} \Big)^{\circ(\mbf n(e)-1)},
	\qea
where we recall that $\Lambda_2 := N^{-1}  \Gamma_w^{o2} \times \Gamma_x^{o2}$. Moreover removing from $\mcal T^\pi$ the edges and internal vertices of $\mathfrak S(e)$ gives the subgraph we obtain assuming $\mbf n(e)=1$. In consequence, our operation is equivalent to replace $Y_{\mbf n(e)}$ by the entry-wise product of the above matrix with $Y_1$. 

Similarly, let $e,e'$ be two edges of $T^{\rho_0}$ that form a simple cycle of length 2. Assume that $\mbf n(e)+\mbf n(e')$ is greater than 2, and denote by $\mathfrak S(e,e')$ a subgraph generated by all internal vertices in the niche of $e$ but one pairing, and all edges attached to it. This graph has one vertex $\mrm{out}$ in $V_1$, on vertex $\mrm{in}$ in $V_2$,  $v_0(e,e'):=\frac {\mbf n(e)+\mbf n(e')-2}2$ internal vertices, and the same configuration of double edges as in $\mathfrak S(e)$ of the previous paragraph. The important fact is that the graph operation $g_{\mbf n(e), \mbf n(e')}  = \big( \mathfrak S(e,e'), \mrm{in}, \mrm{out}\big)$  factorizes
	\eq
		\lefteqn{N^{-v_0(e,e')}g_{\mbf n(e), \mbf n(e')} (\Gamma_w, \Gamma_x)}\\
		&= & \Big( N^{- \frac{ \mbf n(e) -1}2}  g_{\mbf n(e)}  \Big)(\Gamma_w, \Gamma_x) \circ \Big( N^{- \frac{ \mbf n(e') -1}2}  g_{\mbf n(e')}\Big) (\Gamma_w, \Gamma_x),
	\qe
where $\circ$ is the entry-wise product. While removing from $\mcal T^\pi$ the edges and internal vertices of $\mathfrak S(e,e')$ also gives the subgraph we obtain assuming $\mbf n(e)=1$, we can distribute the induced contribution as profiles applied to $e$ and to $e'$ separately.

This motivates the introduction of the following collection of matrices.
\begin{lemma}\label{Lem:DistLin} Let $W=\Gamma_w \circ W'$ and $X=\Gamma_x \circ X'$ be as in our main theorem. Let $\mbf Y^{\mrm{lin}}= \big( Y^{lin}(h) \big)_{h\in \mbb C[y]}$ be the collection of $N_1\times N_2$ profiled matrices such that for any polynomial $h$, 
	$$  Y^{\mrm{lin}}(h)  =   \esp\big[  h' [  \{ \xi \sqrt[\circ]{ \Lambda_2}   \}  ] \big] \circ \Big( \frac{W}{\sqrt N}  \times \frac{X}{\sqrt N} \Big),$$
where  $\Lambda_2 := N^{-1}  \Gamma_w^{o2} \times \Gamma_x^{o2}$ and  $\sqrt[\circ]{ \Lambda_2} := \big(\sqrt{ \Lambda_2(i,j) } \big)_{i,j}$. 
Then $ {\mbf Y^{\mrm{lin}}}$ converges in traffic distribution and for any reference graph $T_{\mbf n} = (V,E,\mbf n)$, with notations as in \eqref{Eq:PreHeuristic} we have
 \eq 
 \tau_N \big[ T_{\mbf n}(   {\mbf Y}^{\mrm{lin}}) \big]  & = &\sum_{\substack{\rho_0 \in \mathcal{P}({V}) \\ T_\mbf n^{\rho_0} \mrm{ \ cactus}}} 
		 \psi_2^{|V_2^{\rho_0}|}  
		 \psi_1^{|V_1^{\rho_0}|} \delta^0\big[\mcal T_{\mbf n}^{\pi_0}(\Gamma_w, \Gamma_x)\big] 				\nonumber\\
		&&  \times \ \psi_0^{c(\rho_0)}  \prod_{e\in E}   \esp\big[ h'_{\mbf n(e)}(\xi) \big]   +o(1).
\qe
\end{lemma}

Note that the expression coincides with \eqref{Eq:PreHeuristic} on cacti whose strong components are simple cycles of length greater than 2. Moreover, the map $(h, W, X )\mapsto  { \mbf Y^{\mrm{lin}}}$ is 3-linear.
\begin{proof} Let $T_{\mbf n}=(V,E,\mbf n)$ be a reference graph and denote by $T_{\mbf 1}=(V,E,\mbf 1)$ the test graph obtained from $T$ by assuming all edges are labeled 1. We denote $\mcal T_{\mbf 1} = (\mcal V_{\mbf 1}, \mcal E_{\mbf 1}, \gamma_{\mbf 1}) := \mcal T_{T_\mbf 1}$.  Let $\mathfrak T=(\mathfrak V, \mathfrak E, \gamma)$ be the test graph in the variables $w$, $x$, and a collection of variables $\mbf r = (r_n)_{n\geq 1}$ obtained from $\mcal T_\mbf 1$ as follows:  for each edge $e$ of $T$, we add a so-called \emph{corrective} edge in $\mcal T_{\mbf 1}$ between the endpoints of $e$ labeled $r_{\mbf n(e)}$. Note that $\mathfrak T$ and $\mcal T_1$ have same vertex set. Then similarly to \eqref{TauToC} for computing $\tau_N[T_\mbf 1( { Y_1})]$, with $\Psi^\pi= \psi_0^{|V_0^\pi|}  \psi_1^{|V_1^\pi|}\psi_2^{|V_2^\pi|}  $ we have
	\eq
	 \lefteqn{\tau_N \big[ T_\mbf n ( \mbf Y^{(\mrm{lin})}) \big]  }\\
		& =&    \sum_{ \pi \in \mathcal{P}(\mathcal{V}_\mbf 1) } \  N^{\eta(\pi)}\Psi^\pi   \delta^0\big[ \mathfrak T(\Gamma_w,\Gamma_x, \mbf R) \big]\delta^0[\mcal T_{\mbf 1}^\pi(W',X')]\big( 1 + o(1) \big).
\qe
where $\eta$ can simply be written
	$$\eta(\pi ) = -1 - \frac{ \mcal E} 2 + |\mcal V^\pi| = \Big(-1 - |\bar {\mcal E}^\pi| + |\mcal V^\pi| \Big) + \Big(|\bar {\mcal E}^\pi| -  \frac{ \mcal E} 2 \Big)$$
 Most of the arguments in the sequel can be deduced from the previous section, but it may be of interest to have an independent sketch of proof. Lemma \ref{eq_graph} and $\delta^0[\mcal T_{\mbf 1}^\pi(W',X')]\neq 0$ imply that  $\mcal T^\pi$ is a cactus whose cycles are of length two (that we call double tree later on). Moreover a pair of double edge have same label, either $w$ or $x$. For each edge $e$ of  $T^{\rho_0}$, we call $w$-neighbor of $e$ the edge $e'$ such that the edges labeled $w$ in their niche are identified, and similarly we define the $x$ neighbors. Two edges are called niche neighbor if they are either $w$ or $x$-neighbors. The definitions imply that the edges of $T$ in a same equivalent class of equivalence for the niche-neighboring relation form a simple cycle in $T^{\rho(\pi)}$. For simple cycles of length greater than 2, their niches induce the usual star-shape subgraph of Section \ref{Focus3} (where the edges labels $\mbf n(e)$ are equal to one), see an example Figure \ref{F04_PCactusWigner}.

  \begin{figure}[h]
    \begin{center}
     \includegraphics[scale=.75]{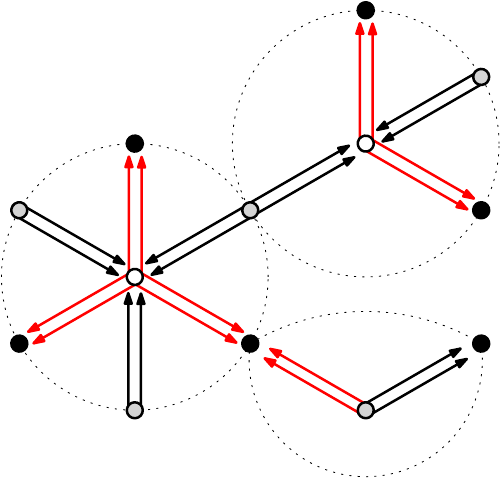}
    \end{center}
    \caption{The test graph that contribute to the injective traffic distribution of $\frac{W' X'}{N}$, where we have encircled the equivalent class for the niche neighboring relation.}
    \label{F04_PCactusWigner}
    \end{figure}

This implies that $T^{\rho(\pi)}$ is a cactus whose  number of strong components is  $c(\rho) =|V_0|^\pi$, and there is a single $\pi$ corresponding to $\rho_0$. Using $\delta^0[\mcal T_{\mbf 1}^\pi(W',X')] = 1$, $ \Psi^\pi = \psi_2^{|V_2^{\rho_0}|} \psi_1^{|V_1^{\rho_0}|}  \psi_0^{c(\rho_0)}$ in the following computation
	\eq
	 \lefteqn{\tau_N \big[ T_\mbf n ( \mbf Y^{(\mrm{lin})}) \big]  }\\
		& =&  \sum_{\substack{\rho_0 \in \mathcal{P}({V}) \mrm{ \ s.t.} \\ T_\mbf n^{\rho_0} \mrm{ \ cactus}}} \      \sum_{\substack{ \pi \in \mathcal{P}(\mathcal{V}_\mbf 1) \mrm{ \ s.t.} \\ \rho(\pi)=\rho_0}}  \Psi^\pi\ \delta^0\big[ \mathfrak T^\pi(\Gamma_w,\Gamma_x, \mbf R) \big]\delta^0[\mcal T_{\mbf 1}^\pi(W',X')] + o(1)\\
		&= &   \sum_{\substack{\rho_0 \in \mathcal{P}({V}) \mrm{ \ s.t.} \\ T_\mbf n^{\rho_0} \mrm{ \ cactus}}}  \psi_2^{|V_2^{\rho_0}|}  
		 \psi_1^{|V_1^{\rho_0}|} \delta^0\big[ \mathfrak T^{\pi_0}(\Gamma_w,\Gamma_x, \mbf R) \big] \times  \psi_0^{c(\rho_0)}   +o(1),
\qe
where $\pi_0$ is any partition of $\mcal V_{\mbf 1}$

On the other hand, recalling that $ \upiota_{(1, 2)}$ is the canonical injection of $\mrm M_{N_1,N_2}(\mbb R) \to \mrm M_{N,N}(\mbb R)$, for any edge $e\in E$ and any split injective map $\phi:V^{\pi_0}\to[N]$, 
	\eq
	\upiota_{(1, 2)}\big( R( h_{\mbf n(e)} )\big)\big( \phi(e) \big) & = &\esp\big[  h_{\mbf n(e)}'(\xi) \big] 
	\big( \upiota_{(1, 2)} (  \Lambda_2  )\big)^{   \circ \frac{\mbf n(e)-1}2 }    \big( \phi(e) \big)\\
	& = & \esp\big[  h_{\mbf n(e)}'(\xi) \big] \times \upiota_{(1, 2)} \big( N^{-v_0}g_{\mbf n(e)}(\Gamma_w, \Gamma_x) \big)\big( \phi(e) \big),
	\qe
where $g_{\mbf n(e)}$ is given in \eqref{Eq:GraphProofLinear}.  Therefore using Lemma \ref{ProfileTrick} on $\mcal T_n$ for each subgraph $\mathfrak S(e)$ and $\mathfrak S(e,e')$ defined in the introduction of this section, we have 
	$$\delta^0\big[ \mathfrak T^{\pi_0}(\Gamma_w,\Gamma_x, \mbf R \big] =   \prod_{e\in E}   \esp\big[ h'_{\mbf n(e)}(\xi) \big] \times \delta^0\big[ \mcal T_n(\Gamma_w,\Gamma_x)\big].$$
Altogether, this proves the asymptotic expression stated in the lemma. Hence the convergence holds for all reference test graph, so $\mbf Y^{\mrm{lin}}$ converges in traffic distribution.
\end{proof}

\subsection{Identification of an additive circular noise}

\subsubsection{Constant profiles}
In this section we assume that all entries of $\Gamma_w$ and $\Gamma_x$ are equal to 1. Since we look for a collection $\mbf Y^{\mrm{per}}$ such that $\mbf Y^{\mrm{lin}} + \mbf Y^{\mrm{per}}+ \mbf B$ as the same limiting traffic distribution as $\mbf Y$, we shall consider the bilinear map 
\eq
	f ( h_1, h_2 )&  =&   \esp \big[  h_{1}(\xi) h_{2}(\xi) \big] -  \esp\big[ h'_{1}(\xi) \big] \times \esp\big[ h'_{2}(\xi) \big],
\qe
obtained from the length 2 cycle contribution of \eqref{Eq:PreHeuristic} minus the expression valid only for higher length cycles. 
The family $(g_n)_{n\geq 0}$ of Hermite polynomials is a basis of $\mbb C[y]^{\mbb N}$ by 
	$$ g_n : y \mapsto   ({-1})^n e^{\frac {y^2} 2} \frac{ \mrm{d}^n}{\mrm d y^n} e^{-\frac{y^2}2}.$$
It is an orthogonal sequence for the standard Gaussian law
	$$\esp[g_n(\xi) g_m(\xi)]=\delta_{n,m} n!, \quad \forall n,m\geq 0.$$
The Leibniz formula implies the  formula 
	$$  \frac{ \mrm{d}^{n+1}}{\mrm d y^{n+1}} e^{-\frac{y^2}2} = -y  \frac{ \mrm{d}^n}{\mrm d y^n} e^{-\frac{y^2}2}-2n  \frac{ \mrm{d}^{n-1}}{\mrm d y^{n-1}} e^{-\frac{y^2}2}$$
 which yields the identity $g'_n(y) =n g_{n-1}(y), \quad \forall n\geq 1.$
 Moreover the orthogonality relation implies  $\esp[ g_n(\xi)] = \esp[ g_n(\xi)g_0(\xi)]=\delta_{n,0}$ for any $n\geq 1$ and so $\esp[g'_n(\xi)] = n \esp[g_{n-1}(\xi)] = \delta_{n,1}$. We therefore have for all $n,m\geq 1$
	\eq
		  f(g_n, g_m)  & = &  n! \delta_{n,m} - \delta_{n,1}\delta_{m,1}.
	\qe
We can hence propose the following collection of matrices. We call \emph{double-tree} a cactus whose simple cycles are all of length two.

\begin{lemma}\label{Lem:DistPerConstant} Let $\mbf Y^{per} = \big( Y^{per}(h) \big)_{h \in \mbb C[y]}$ be a collection of $N_1\times N_2$ random matrices such that the map $h\mapsto Y^{per}(h)$ is linear and
\begin{itemize}
	\item  $Y^{per}(g_0) = Y^{per}(g_1)=0$,
	\item  the matrices $\sqrt{\psi_0  n ! N }^{-1}Y^{per}(g_n),n\geq 2$ labeled by Hermite polynomials are i.i.d. with i.i.d.  real standard Gaussian variables. 
\end{itemize}
Then $\mbf Y^{per}$ converges in traffic distribution and for any reference test graph $T_{\mbf n}$ whose edges are labeled by integer greater than 1,
\eqa
	\lefteqn{ \tau_N \big[ T_{\mbf n}( \mbf Y^{per}) \big] }\label{Eq:DistPerConstant}\\
		& = &  \sum_{\substack{\rho_0 \in \mathcal{P}({V}) \\ T_{\mbf n}^{\rho_0} \mrm{\ double  \ tree}}} 
		 \psi_2^{|V_2^{\rho_0}|}  
		 \psi_1^{|V_1^{\rho_0}|} 	 \psi_0 ^{c(\rho_0)}			\nonumber\\
		 & & \times  \prod_{C\in \mcal C(\rho_0)}  \Big(   \esp \big[  h_{n(C)(\xi)} h_{n'(C)}(\xi) \big] -   \esp \big[  h'_{n(C)(\xi)} \big]  \esp \big[  h'_{n'(C)(\xi)} \big] \Big)    \nonumber    \\
 	\qea
where notations are as in Lemma \ref{Lem:PreHeuristic}.
\end{lemma}

\begin{example}For instance, one can compute the first odd Hermite polynomials $g_1(y) = y$, $g_3(y)=y^3-3y$ and $g_5(y) = y^5-10y^3 +14y$. Hence we have for the power functions $h_3 = g_3+3g_1$ and $h_5 = g_5 +10 g_3 +16 g_1$. So we can write $Y^{per}(h_3) = \sqrt{6} Z_1 $, $Y^{per}(h_5) = 10\sqrt{6}  Z_1 +2\sqrt{30} Z_2$, where $Z_1$ and $Z_2$ are independent matrices with i.i.d. real centered Gaussian variable of variance $\psi_0$.
\end{example}

\begin{remark} Let consider $\mbf Y^{\mrm{lin}}$ defined in Lemma \ref{Lem:DistLin}. If the profile matrices $\Gamma_w, \Gamma_x$ are constant equal to one, then so is the matrix $\sqrt[\circ]{ \Lambda_2}$. Hence the property $\esp[g'_n(\xi)] = \delta_{n,1}$ implies that $Y^{\mrm{lin}}(g_n) = 0$ if $g_n$ is the Hermite polynomial of order $n\geq 2$. The same reasoning shows that, for $\mbf B$ defined in Lemma \ref{Lem:DistB}, then $B(g_n)=0$ unless $n=3$. This is not longer true when $\sqrt[\circ]{ \Lambda_2}$ is not constant equal to one.
\end{remark}

\begin{proof} Let $T$ be a test graph whose edges are labeled by Hermite polynomials of order greater than 1.  The associated rectangular matrices are independent with i.i.d. entries, for which the computation of the traffic distribution is similar to \cite[Chapter 3]{Male2020}
\eq
	\lefteqn{ \tau_N \big[ T_{\mbf g}( \mbf Y^{per}) \big] }\\
	& = & \sum_{\rho_0 \in \mcal P(V) } \frac 1 N \sum_{\substack{\phi:V^{\rho_0} \to [N] \\ \mrm{injective}}} \esp\bigg[ \prod_{e \in E} \upiota_{1,2}\Big( Y^{per} (g_{e})\Big)\big( \phi(e) \big) \bigg]\\
	 & =&  \sum_{\rho_0 \in \mcal P(V) }  \sum_{\substack{\phi:V^{\rho_0} \to [N] \\ \mrm{injective}}} N^{\eta(\rho_0)}\big(1+o(1)\big) \psi_1^{|V_1^{\rho_0}|}\psi_2^{|V_2^{\rho_0}|} \delta^0\big[ T_{\mbf g}( \sqrt N \mbf Y^{per}) \big]
\qe
where $\eta(\rho_0) =-1-\frac{|E|}2 + |V^{\rho_0}|$. By Lemma \ref{eq_graph} we get
\eq
	 \tau_N \big[ T_{\mbf g}( \mbf Y^{per}) \big] & = & \sum_{\substack{\rho_0 \in \mcal P(V) \mrm{ \ s.t. }\\ T^{\rho_0} \mrm{ \ double \ tree}}}  \sum_{\substack{\phi:V^{\rho_0} \to [N] \mrm{injective}}}  \psi_1^{|V_1^{\rho_0}|}\psi_2^{|V_2^{\rho_0}|}\\
	 & & \ \times    \prod_{\{e,e'\} \in \bar E}  \esp\big[ y^{(g_e)}_{p,q} y^{(g_{e'})}_{p,q} \big]+ o(1)
\qe

In the above sum, $\bar E$ denotes the skeleton of $E$ and the product is over all elements of $\bar E$ that are denoted $\{e,e'\}$ where $e$ and $e'$ are two edges of $T_{\mbf g}$ that are identified by $\rho$. By definition of the matrices $Y^{per}$ and by orthogonality of the Hermite polynomials, we have $ \esp\big[ y^{(g_e)}_{p,q} y^{(g_{e'})}_{p,q} \big] =    \delta_{g_e , g_{e'}} \times  \psi_0 n(\mbf e)! = \psi_0 \esp\big[ g_e(\xi)g_{e'}(\xi)\big]$. We hence obtain the asymptotic formula 
\eq
	\lefteqn{ \tau_N \big[ T_{\mbf g}( \mbf Y^{per}) \big]}\\
	 & = &
	  \sum_{\substack{\rho_0 \in \mathcal{P}({V}) \\ T_{\mbf n}^{\rho_0} \mrm{\ double  \ tree}}} 
		 \psi_2^{|V_2^{\rho_0}|}  
		 \psi_1^{|V_1^{\rho_0}|} 	 \psi_0 ^{c(\rho_0)}  \prod_{C\in \mcal C(\rho_0)} \esp \big[ g_{n(C)}(\xi) g_{n'(C)}(\xi) \big].
\qe

Since $ \esp \big[  g'_n(\xi) \big]=0$ for all $n>1$, the asymptotic formula \eqref{Eq:DistPerConstant} is valid for all test graphs labeled by Hermite polynomials of positive order. Indeed, if there is an edge labeled $g_1$ then, since  $ \esp \big[ g_1(\xi) g_n(\xi) \big] -   \esp \big[  g'(\xi) \big]  \esp \big[  g'_n(\xi) \big]  = \esp \big[ \xi g_n(\xi) \big] -  \esp \big[  g'_n(\xi) \big] =0$ then \eqref{Eq:DistPerConstant} vanishes, and otherwise the expressions coincides. Since the Hermite polynomials form a basis of $\mbb C[y]$ and the map $h\mapsto Y^{per}(h)$ is linear, this proves the convergence in traffic distribution of $\mbf Y^{per}$. Finally, since \eqref{Eq:DistPerConstant} is a multilinear function of the labels, the formula is also valid for all test graphs labeled in $\mbb C [y]$. 
\end{proof}

In the rest of this section, we emphasis a property that we use later. We say that a family $\mbf A$ that converges in traffic distribution satisfies the \emph{asymptotic factorization property} whenever for any test graph $T_1\etc T_L$, we have
	$$\esp \Big[ \prod_{\ell=1}^L \frac 1 N \Tr \big[T(\mbf A ) \big] \Big] = \prod_{\ell=1}^L \tau_N\big[ T(\mbf A) \big] +o(1).$$
	
\begin{lemma} The couple $(W'/\sqrt{N}, X'/\sqrt N)$ and the collection ${\mbf Y^{\mrm{per}}}'$ defined in Lemma \ref{Lem:DistLin} and Lemma \ref{Lem:DistPer} satisfy the asymptotic factorization property. 
\end{lemma}

\begin{proof} The proof follows from minor modification of the convergence $ \tau_N\big[ T(\mbf Y^{\mrm{lin}}) \big]$ by considering a unconnected graph $T$ and normalizing the trace by $N^{-c}$ where $c$ is the number of connected components. The independence of the matrix entries shows that we can factorize the contributions of each connected component. 
\end{proof}

\subsubsection{General profiles}

We consider a collection of matrices  as in Lemma \ref{Lem:DistPerConstant} that we denote ${\mbf Y^{per}}' = \big({Y^{per}}'(h)\big)_{h\in \mbb C[y]}$. 
We look for a collection of matrices ${\mbf Y^{per}} = \big( Y^{per}(h) \big)_{h\in \mbb C[y]}$ of the form
	\eqa\label{Eq:Proposition Yper}
		 {  Y^{per}}(h) = \tilde R(h) \circ {\mbf Y^{per}}'(h)
	\qea
for  $\tilde R(h)$ chosen in order to match the remaining terms.

Section \ref{Sec:Ylin} gives an expression of  the traffic distribution of a couple of $N_1\times N_2$  profiled  Gaussian matrix $Z_1 =   R_1  \circ Z'$, $Z_2 = R_2 \circ Z'$, where $\sqrt N Z'$ has i.i.d. standard real Gaussian entries. Their injective traffic distribution is supported on double trees, and the contribution of a double edge with labels $z_1$ and $z_2$ is as follows: for the graph monomial $g$ with two vertices $\mrm{out}\in V_1$, $\mrm{in} \in V_2$, and one double edge labeled $z_1$ and $z_2$ from $\mrm{in}$ to $\mrm{out}$, we have
	$$g(   R_1,R_2 ) =  R_1 \circ R_2.$$
Of independence interest, note that $Ng(   R_1,R_2 )$ is the matrix of the co-called $\mcal R$-transform over the diagonal computed by Shlyakhtenko \cite{SHL}.

On the other hand, let us $S$ be a double edge formed by two edges $e,e'$ of $T^{\rho_0}$ as in Section \ref{Focus2}. Its vertices are $ \mrm{in}\in V_2,  \mrm{out}\in V_1$ and a set $\mcal V_0(S)$ of $\frac{ \mbf n(e) + \mbf n(e')}2$ internal vertices, with double edges between internal and non internal vertices as usual. The associated graph monomial $g_{\mbf n(e)+ \mbf n(e')}=(S,\mrm{in}, \mrm{out})$ satisfies
	\eq
		\lefteqn{ \frac 1 N g_{\mbf n(e)+ \mbf n(e')}(\Gamma_w, \Gamma_x) }\\
	 & = & \bigg[  \Big( \frac 1 N  \sum_{k=1}^{N_0} \Gamma_w^{\circ 2}(i,k) \Gamma_w^{\circ 2} (k,j)\Big)^{ \frac{ \mbf n(e) + \mbf n(e')}2} \bigg]_{i,j} =  \sqrt[\circ]{ \Lambda_2}^{\circ  \mbf n(e)} \circ    \sqrt[\circ]{ \Lambda_2}^{\circ  \mbf n(e')} .
	\qe
We can hence propose the following collection of matrices, using the simple relation $y^n h'_n(x) = y h'_n(xy)$ for $h_n:x\mapsto x^n$.

\begin{lemma}\label{Lem:DistPer} Let ${\mbf Y^{per}} = \big( Y^{per}(h) \big)_{h\in \mbb C[y]}$ be the collection of matrices 
	$$Y^{per}(h)  =  h \big[ \big\{ \sqrt[\circ]{ \Lambda_2} \big\}  \big]  \circ {Y^{per}}'(h),$$
where the collection ${\mbf Y^{per}}'$ is as in Lemma \ref{Lem:DistPerConstant}. Then ${\mbf Y^{per}} $ converges in traffic distribution, and for any reference test graph $T_{\mbf n}$, we have
\eq
	\lefteqn{ \tau_N \big[ T_{\mbf n}( \mbf Y^{per}) \big]}\\
	 & = &
	  \sum_{\substack{\rho_0 \in \mathcal{P}({V}) \\ T_{\mbf n}^{\rho_0} \mrm{\ double  \ tree}}} 
		 \psi_2^{|V_2^{\rho_0}|}  
		 \psi_1^{|V_1^{\rho_0}|} \delta^0\big[ \mcal  T_{\mbf n}^{\pi_0}(\Gamma_w, \Gamma_x)\big] 	 \psi_0 ^{c(\rho_0)}  \prod_{C\in \mcal C(\rho_0)} f(h_{\mbf n(e)}, h_{\mbf n(e')}),
\qe
where we recall that $f ( h_1, h_2 ) =   \esp \big[  h_{1}(\xi) h_{2}(\xi) \big] -  \esp\big[ h'_{1}(\xi) \big] \times \esp\big[ h'_{2}(\xi) \big]$.
\end{lemma}

\begin{proof}Let $T=(V,E,\mbf n)$ be a reference test graph and denote $\tilde {\mbf R}=\big(  \Lambda_2^{\circ\frac n 2}\big)_{n\geq 1}$. The collection of matrices ${\mbf Y^{\mrm{per}}}'$ is invariant in law by left and right multiplication by permutation matrices,  so we can factorize the profiles under the injective trace
\eq
	 \lefteqn{ \tau_N \big[ T_{\mbf n}( \mbf Y^{per}) \big]}\\
	  & = & \sum_{\rho_0\in \mcal P(V) } \delta^0\big[ T^{\rho_0}_{\mbf n} ( \tilde{  \mbf R} )\big] \times  \tau^0_N \big[ T^{\rho_0}_{\mbf n}( {\mbf Y^{per}}') \big].
\qe
Therefore we can use the expression  $ \tau^0_N \big[ T_{\mbf n}( {\mbf Y^{per}}') \big]$ from the previous case.
\eq
	 \lefteqn{ \tau_N \big[ T_{\mbf n}( \mbf Y^{per}) \big]}\\
	  & = & \sum_{\substack{\rho_0 \in \mathcal{P}({V}) \\ T_{\mbf n}^{\rho_0} \mrm{\ double  \ tree}}} 
		 \psi_2^{|V_2^{\rho_0}|}  
		 \psi_1^{|V_1^{\rho_0}|} 	 \psi_0 ^{c(\rho_0)}	 \delta^0\big[ T^{\rho_0}_{\mbf n} (  \tilde{\mbf R} )\big] 		\nonumber\\
		 & & \times  \prod_{C\in \mcal C(\rho_0)}  \Big(   \esp \big[  h_{n(C)(\xi)} h_{n'(C)}(\xi) \big] -   \esp \big[  h'_{n(C)(\xi)} \big]  \esp \big[  h'_{n'(C)(\xi)} \big] \Big).
\qe
On the other hand, recalling that $ \upiota_{(1, 2)}$ is the canonical injection of $\mrm M_{N_1,N_2}(\mbb R) \to \mrm M_{N,N}(\mbb R)$, for any double edge $\bar e = \{ e, ,e'\}$ of $T_{\mbf n}^{\rho_0}$ and any split injective map $\phi:V^{\rho_0}\to[N]$, the definition of $\tilde {\mbf R}$ implies
\eq
	\lefteqn{ \upiota_{(1, 2)} (  \tilde{ R}_{n(e)}  )\big( \phi(  e) \big)    \times \upiota_{(1, 2)}(  \tilde{ R}_{n(e')}  )\big( \phi(  e') \big) } \\
	& = & \upiota_{(1, 2)}\big(  \Lambda_2^{\circ \mbf n(e)} \big) \big( \phi(  e) \big) \times  \upiota_{(1, 2)}\big(  \Lambda_2^{\circ \mbf n(e')} \big) \big( \phi(  e') \big) \\
	&= &  
	  \upiota_{(1, 2)}\Big( N^{-\frac{\mbf n(e)+\mbf n(e)}2}g_{\mbf n(e)+\mbf n(e')}(\Gamma_w, \Gamma_x) \Big)\big( \phi( e) \big),
	\qe
with  $g_{\mbf n(e)+\mbf n(e')}$ as in the beginning of the section. Therefore using Lemma \ref{ProfileTrick} for each double edge, we have 
	$$\delta^0\Big[ T^{\rho_0}_{\mbf n}\big( \sqrt[\circ]{ \Lambda_2}  \circ  \mbf R(h) \big)\Big] 	 =  \delta^0\big[\mcal T^{\pi_0}_{\mbf n}(\Gamma_w, \Gamma x) \big].$$
	Altogether this proves the expected asymptotic formula.
\end{proof}

\subsection{Conclusion}

We shall now use the asymptotic traffic independence principle for the collections $\mbf Y^{\mrm{lin}},\mbf Y^{\mrm{per}}$ and $\mbf B$. The drawback of our presentation is that since we have variance profiles we cannot use existing theorems to conclude. Although it is easy to use this theorem we are going to repeat the arguments of the three last section.

Let $\mbf Z_1$, $\mbf Z_2$ and $\mbf Z_3$ be three independent families of rectangular matrices indexed by some set $J$, that converges in traffic distribution and satisfy the asymptotic factorization property. Assume furthermore the  families are \emph{bi-permutation invariant}, that is $\mbf Z_j$ has the same law as the collection $\big(U \mbf Z_jV\big)_{j\in J}$ for any permutation matrices $U$ and $V$ of appropriate size, for $j=1,2,3$.

Then the asymptotic traffic independence theorem for rectangular matrices \cite{Zit24} proves that $(\mbf Z_1, \mbf Z_2, \mbf Z_3)$ converges in traffic distribution. Let us recall the sketch of the proof. Let $\mcal T$ be a testing graph in three families of variables $\mbf z_1, \mbf z_2, \mbf z_3$  with vertex set $\mcal V = V_0 \sqcup V_1\sqcup V_2$. Let $\pi$ be a split partition of its vertex set and denote for any $i=1,2,3$ by $ \mcal  T_i$ is the graph obtained from $ \mcal  T$ by removing edge that are not in $\mbf z_i$. Its vertex set is denoted $\mcal V_i = V_{0,i} \sqcup V_{1,i}\sqcup V_{2,i}$. Setting $(N)_{\mcal T^\pi} = \prod_{j=0,1,2} \big(N_{j} \big)_{| \mcal  V_{j}^{\pi}|}$, we have as before
\eq
	\tau_N^0\big[ \mcal T^{\pi}( \mbf Z_1, \mbf Z_2, \mbf Z_3) \big] & = & N^{-1} (N)_{\mcal T^\pi} \times   \delta^0 \big[ \mcal T^{\pi}( \mbf Z_1, \mbf Z_2, \mbf Z_3) \big]
\qe
and reciprocally
\eq
	\lefteqn{ \delta^0 \big[ \mcal T^{\pi}( \mbf Z_1, \mbf Z_2, \mbf Z_3) \big] = \delta^0 \big[ \mcal T^{\pi}( \mbf Z_1) \big]  \delta^0 \big[  \mcal T^{\pi}( \mbf Z_2) \big]   \delta^0 \big[ \mcal T^{\pi}( \mbf Z_3) \big] }\\
	  & = & (N)_{\mcal T_1^\pi}^{-1}(N)_{\mcal T_2^\pi}^{-1} (N)_{\mcal T_3^\pi}^{-1} \times \esp\big[  \prod_{S \in \mcal CC( \mcal T^{\pi})}  \Tr^0 S(\mbf Z_{i_S}) \big],
\qe
where the product is over the union all connected components $S$ of $ \mcal T_1^{\pi}$, $ \mcal T_2^{\pi}$ and $ \mcal T_3^{\pi}$, and $i_S$ denotes the edge labels in $1,2,3$ of matrices associated to $S$. 

We therefore get, setting $\mbf \Psi^{\pi} = \prod_{j=0,1,2} \psi_j^{| \mcal V_j^{\pi}| - \sum_{i=1,2,3} |  \mcal  V_{j,i}^{\pi}|}  $ we get 
\eq
 \tau_N^0\big[ \mcal  T^{\rho}( \mbf Z_1, \mbf Z_2, \mbf Z_3) \big]  	& = & N^{\eta(\pi)} \big( 1 + o(1) \big) \mbf \Psi^{\pi}  \prod_{S \in \mcal CC( \mcal T^{\pi})} \tau_N^0\big[  S(\mbf Z_{i_S}) \big]
\qe
for some $\eta(\pi)$ whose expression can be made explicite from the above computation. The important properties are that 
\begin{enumerate}
	\item $\eta(\pi)\leq 0$ with equality if and only if a certain graph, called the graph of colored component of $\mcal T^{\pi}$, is a tree \cite{Male2020,Zit24}
	\item if the elements of $\mcal{CC}(\mcal T^{\pi})$ are pseudo-cacti, then the graph of colored component of $\mcal T^{\pi}$ is a tree if and only if $\mcal T^{\pi}$ is a pseudo-cactus whose strong components have edges labels associate to a single family among $\mbf Z_1, \mbf Z_2, \mbf Z_3$ \cite{CDM24}.
\end{enumerate}
Assuming that for any test graph $S$ that is not a pseudo-cactus, $\Nlim \tau_N^0\big[ S(\mbf Z_i) \big]  =0$ for $i=1,2,3$, we get
\eq
	\tau_N\big[\mcal  T( \mbf Z_1, \mbf Z_2, \mbf Z_3) \big] & =& \sum_{ \substack{ \pi \in \mcal P(\mcal V) \mrm{ \ s.t.} \\ \mcal T^{\pi} \mrm{ \ w.c. \, p.-cactus}}}  \mbf \Psi^{\pi}  \prod_{S \in \mcal {CC}( \mcal T^{\pi})} \tau^0_N\big[ S(\mbf Z_{i_S})\big] +o(1)
\qe
where ''$\mrm{w.c. \, p-cactus}$`` is a shortcut for \emph{well-colored} pseudo-cactus, meaning that all edges of each simple cycle of the cactus $\mcal T^{\pi}$ are labeled either by $\mbf Z_1$, $\mbf Z_2$ or $\mbf Z_3$. 

The collections of matrices  $\mbf Y^{\mrm{lin}}, \mbf Y^{\mrm{per}}$ and $\mbf B$ are not bi-permutation invariant when the profiles are not constant. But they are defined by applying profiles to bi-permutation collections of independent matrices and we can use this property. 

More precisely, with $\mbf Z_1$, $\mbf Z_2$, $\mbf Z_3$ and the test graph $\mcal T$ given as above, assume we are also given a collection of matrices with bounded entries $\mbf \Gamma$ and a test graph $\mcal T_+$ obtained from $\mcal T$ by adding a set $\mcal E_+$ of edges labeled for matrices in $\mbf \Gamma$, but without adding vertices. Let $\mcal T_-$ be the union of test graphs obtained from $\mcal T_+$ by removing the edges that are not standing for a matrix in $\mbf \Gamma$. Since $\mcal T$ and $\mcal T_+$ have same vertex set we have
\eq
	\lefteqn{\tau_N\big[ \mcal T_+^{\rho}( \mbf Z_1, \mbf Z_2, \mbf Z_3, \mbf \Gamma) \big]  }\\
	& = & \sum_{ \substack{ \rho \in \mcal P(\mcal V) }}  N^{-1}  \big(N_{1} \big)_{| \mcal  V_{1}^{\rho}|}  \big(N_{2} \big)_{|  \mcal V_{2}^{\rho}|}   \delta^0 \big[ \mcal T^{\pi}( \mbf Z_1, \mbf Z_2, \mbf Z_3) \big]\delta^0 \big[ \mcal T_-^{\pi}( \mbf  \Gamma) \big] \\
	& =& \sum_{ \substack{ \pi \in \mcal P(\mcal V) \mrm{ \ s.t.} \\ \mcal T^{\pi} \mrm{ \ w.c. \, p-cactus}}}  \mbf \Psi^{\pi}  \delta^0 \big[ \mcal T_-^{\pi}(\mbf  \Gamma) \big] \prod_{S \in \mcal {CC}( \mcal T^{\pi})} \tau^0_N\big[ S(\mbf Z_{i_S})\big] +o(1)
\qe
Informally, we can factorize the profile contribution under the injective trace. We recall the definitions and set the following notations
	\eq
	 Y^{\mrm{lin}}(h) & :=&  \esp\big[  h' [  \{ \xi \sqrt[\circ]{ \Lambda_2}   \}  ] \big]   \circ  \Big( \frac{W}{\sqrt N}  \times \frac{X}{\sqrt N} \Big)\\
	 & =: & \Gamma_1(h) \circ \big( (\Gamma_w \circ  \frac{W'}{\sqrt N} ) \times (\Gamma_x \circ  \frac{X'}{\sqrt N})\big) \\
	 Y^{\mrm{per}}(h) &:= & h \big[ \big\{ \sqrt[\circ]{ \Lambda_2} \big\}  \big]  \circ {Y^{per}}'(h) =: \Gamma_2(h) \circ {Y^{per}}'(h),\\
	 B(h) & = &   
		 \frac {m_w^{(3)} m_x^{(3)}}{6N}  \Lambda_3\circ  \esp\bigg[  h^{'''}\Big[ \Big\{ \xi \sqrt[\circ]{ \Lambda_2}  \Big\} \Big] \bigg] =: \Gamma_3(j) \circ \mbb J_N
	 \qe
 where $\mbb J_N$ is the matrix whose entries are $\frac 1 N$. We set $\mbf Z_1 = (W'/\sqrt N, X'/\sqrt N)$, $\mbf Z_2 = {\mbf Y^{per}}'$, $\mbf Z_3 = ( \mbb J_N)$ and the collection $\mbf \Gamma $ consisting in $\Gamma_w, \Gamma_x$ and the matrices $\Gamma_1(h),   \Gamma_2(h), \Gamma_3(h)$ for all $h\in \mbb C [y]$. For any test graph $T$ in three collections of variables
\eq
	\tau_N\big[T( \mbf Y^{\mrm{lin}}, \mbf Y^{\mrm{per}} , \mbf B)\big] & = & \tau_N\big[\mcal T_+( \Gamma, \mbf Z_1, \mbf Z_2, \mbf Z_3)\big] 
\qe
where $\mcal T_+$ is obtained from $T$
\begin{itemize}
	\item adding for each edge labeled of $T$ associated to $Y^{\mrm{per}}(h)$ and edge with label $\Gamma_2(h)$ and same endpoints,
	\item adding for each edge labeled of $T$ associated to $B(h)$ and edge with label $\Gamma_3(h)$ and same endpoints,
	\item adding for each edge labeled of $T$ associated to $Y^{\mrm{lin}}(h)$ and edge with label $\Gamma_2(h)$ and same endpoints, 
	\item replacing each edge of $T$ associated to $Y^{\mrm{lin}}(h)$ by a niche with one internal vertex, one edge labeled $w'$ and one edge labeled $x'$
	\item adding for each edge in labeled $w'$ and edge with label $\Gamma_w$ and same endpoints, and for each edge in labeled $x'$  and edge with label $\Gamma_x$ similarly. 
\end{itemize}

So we can apply the above observation: with same notations as above
\eq
	\lefteqn{ \tau_N\big[\mcal T_+( \Gamma, \mbf Z_1, \mbf Z_2, \mbf Z_3)\big]  }\\
	& =& \sum_{ \substack{ \pi \in \mcal P(\mcal V) \mrm{ \ s.t.} \\ \mcal T^{\pi} \mrm{ \ w.c. \, p-cactus}}}   \mbf \Psi^{\pi} \delta^0 \big[ \mcal T_-^{\pi}(\mbf  \Gamma) \big] \prod_{S \in \mcal {CC}( \mcal T^{\rho})} \tau^0_N\big[ S(\mbf Z_{i_S})\big] +o(1).
\qe

We have for any test graph $S$
	\eq
		 \tau^0_N\big[ S(\mbf Z_3)\big] & \limN & \one\Big( S \mrm{ \ is \ a \ tree} \Big),\\
		  \tau^0_N\big[ S(\mbf Z_2)\big] & \limN & \one\Big( S \mrm{ \ is \ a \ double \ tree} \Big) \prod_{C\in \mcal {SC}(S)} f\big( h_{n(S)}, h_{n'(S)}\big) \\
		  \tau^0_N\big[ S(\mbf Z_1)\big] & \limN & \one\Big( S \mrm{ \ is \ a \ w.c.-double \ tree} \Big),\\
	\qe
where in the second formula $h_{n(S)}, h_{n'(S)}$ are the edges labels in the double edge $S$ and in the third one ''$\mrm{w.c.-double \ tree}$`` means that for each doubles edge, both edges  labels are $w$ or are $x$. 

Let as before $\rho(\pi)$ be the restriction of $\pi$ to the vertices of $\mcal V$ in $V_1 \sqcup V_2$. The previous niche-neighbor argument shows that $T^{\rho}$ is a cactus such that for the simple cycles of $T^{\rho}$ labeled by the edges associated to $\mbf Y^{\mrm{lin}}$, the edges of their niche forms the star-shape test graph of section \ref{Focus3} when all edges labels are 1. We therefore have, when $T$ is a reference test graph
\eq
	\lefteqn{\tau_N\big[T( \mbf Y^{\mrm{lin}}, \mbf Y^{\mrm{per}} , \mbf B)\big]  }\\
	& =& \sum_{ \substack{ \rho_0 \in \mcal P(  V) \mrm{ \ s.t.} \\ T^{\rho_0} \mrm{  pseudo-cactus}}}   \mbf \Psi^{\rho} \prod_{S \in \mcal {CC}_1(  T^{\rho_0})}  \one\Big( \mrm{the \ label \ is \ in \ } \mbf z_3 \Big)\\
	& & \times  \prod_{S \in \mcal {CC}_2(  T^{\rho_0})} \bigg(  \one\Big( \mrm{the \ labels \ are \  in \ } \mbf z_2 \Big) \psi_0  \prod_{C\in \mcal {SC}(S)} f\big( h_{n(S)}, h_{n'(S)}\big)\bigg)\\
	& & \times  \prod_{S \in \mcal {CC}_3(  T^{\rho_0)} } \one\Big( \mrm{the \ labels \ are \ in \ } \mbf z_1 \Big)  \psi_0	\times  \delta^0 \big[ \mcal T_-^{\pi_0}(\mbf  \Gamma) \big] +o(1).
\qe
where $\pi_0$ is the only partition of $\mcal V$ such that $\rho(\pi) = \rho_0$ and $\mcal T^{\pi}$ is a pseudo-cactus with simple cycles of length two. Going back to the definition as the profile matrices as in the previous sections shows
\eq
	\lefteqn{\tau_N\big[T( \mbf Y^{\mrm{lin}}, \mbf Y^{\mrm{per}} , \mbf B)\big]  }\\
	& =& \sum_{ \substack{ \rho_0 \in \mcal P(  V) \mrm{ \ s.t.} \\ T^{\rho_0} \mrm{  pseudo-cactus}}}  \delta^0 \big[ \mcal T_n^{\pi_0}(\mbf  \Gamma) \big]   \mbf \Psi^{\rho} \\
	& & \times  \prod_{\substack{S \in \mcal {CC}_1(  T^{\rho_0})\\ \mrm{with \ label \ in \ } \mbf z_3}} \left (  \frac{ m_w^{(3)} m_x^{(3)}}6\right)^{c_1(\rho_0)} \prod_{S\in \mcal C_1(\rho_0)}  \esp\big[   h'''_{n(S)}(\xi) \big]\\
	& & \times  \prod_{\substack{S \in \mcal {CC}_2(  T^{\rho_0})\\ \mrm{with \ label \ in \ } \mbf z_2}}  \psi_0  \prod_{C\in \mcal {SC}(S)} f\big( h_{n(S)}, h_{n'(S)}\big)\\
	& & \times  \prod_{\substack{S \in \mcal {CC}_3(  T^{\rho_0)} \\ \mrm{with \ label \ in \ } \mbf z_1}} \psi_0 \prod_{\ell=1}^{L_S}\esp[ h'_{n_\ell(S)}(\xi) ]   	+o(1).
\qe

Hence, if $T =(V,E, \gamma)$ is a reference test graph in a single collection $\mbf z$, denoting for any $\theta:E \to \{1,2,3\}$ by $T_\theta$ the test graph obtained from $T$ by changing for each edge $e$ its label $\gamma(e) = z(h)$ into $z_{\theta(e)}(h)$, we have
	\eq
		\tau^0_N\big[T^{\rho_0}( \mbf Y^{\mrm{lin}} +  \mbf Y^{\mrm{per}} + \mbf B)\big] & = & \sum_{\theta :E \to \{1,2,3\}} \tau^0_N\big[T^{\rho_0}_\theta( \mbf Y^{\mrm{lin}}, \mbf Y^{\mrm{per}} , \mbf B)\big].
	\qe
If $S$ is a double edge of $T^{\rho_0}$, either $\theta$ attributes  labels in $\mbf z_1$ for both edges, or labels in $\mbf z_2$. All other contributions factorizing, these two termes adds up to give the expected formula. This proves that any trace of $\mbf Y^{\mrm{lin}} +  \mbf Y^{\mrm{per}} + \mbf B$ in a reference test graph labeled by odd polynomials satisfies the same asymptotic formula than the collection of Pennington-Worah matrices. Hence they have the same limiting traffic distribution, which conclude the proof of the main theorem. 
\bibliographystyle{alpha}
\bibliography{biblio.bib}

\end{document}